\documentclass[12pt,fleqn]{article}
\usepackage{amsfonts}
\usepackage{mathrsfs}
\usepackage[dvips]{graphics}
\usepackage[dvips]{color}
\usepackage{amsthm}
\usepackage[]{amsmath}
\usepackage{CJK}
\newtheorem{proposition}{Proposition}[section]
\newtheorem{lemma}[proposition]{Lemma}
\newtheorem{definition}[proposition]{Definition}
\newtheorem{Theorem}[proposition]{Theorem}
\newtheorem{corollary}[proposition]{Corollary}
\begin{document}
\begin{CJK*}{GBK}{song}

\centerline{\Large{\textbf{Gap Sequence of Factors of Fibonacci Sequence}}}

\vspace{0.1cm}

\centerline{Huang Yuke\footnote[1]{Department of Mathematical Sciences, Tsinghua University, Beijing, 100084, P. R. China.}$^,$\footnote[2]{E-mail address: hyg03ster@163.com.}
~~Wen Zhiying\footnote[3]{Department of Mathematical Sciences, Tsinghua University, Beijing, 100084, P. R. China.}$^,$\footnote[4]{E-mail address: wenzy@tsinghua.edu.cn(Corresponding author).}}

\vspace{1cm}

\centerline{\textbf{\large{ABSTRACT}}}

\vspace{0.4cm}

\noindent Let $\omega$ be a factor of Fibonacci sequence $F_\infty=x_1x_2\cdots$, then it appears in the sequence infinitely many times. Let $\omega_p$ be the $p$-th appearance of $\omega$ and $\nu_{\omega,p}$ be the gap between $\omega_p$ and $\omega_{p+1}$.
In this paper, we discuss the structure of the gap sequence $\{\nu_{\omega,p}\}_{p\geq1}$, we first introduce the singular kernel word $sk(\omega)$ for any factor $\omega$ of $F_\infty$
and give a decomposition of $\omega$ with respect to $sk(\omega)$. Using the singular kernel and the decomposition, we prove the gap sequence $\{\nu_{\omega,p}\}_{p\geq1}$
has exactly two different elements $\{\nu_{\omega,1},\nu_{\omega,2}\}$ and determine the expressions of gaps completely, then we prove that the gap sequence over the alphabet $\{\nu_{\omega,1},\nu_{\omega,2}\}$ is still a Fibonacci sequence. Finally, we introduce the spectrum for studying some typical combinatorial, using the results above, we determine completely the spectrums.

\vspace{0.4cm}

\setcounter{section}{1}

\noindent\textbf{\large{1.~Introduction}}

\vspace{0.4cm}

Let $\mathcal{A}=\{a,b\}$ be a binary alphabet. Let $\mathcal{A}^\ast$ be the set of finite words on $\mathcal{A}$ and $\mathcal{A}^{\mathbb{N}}$ be the set of one-sided infinite words. The elements of $\mathcal{A}^\ast$ are called words or factors, which will be denoted by $\omega$. The neutral element of $\mathcal{A}^\ast$ is called the empty word, which we denote by $\varepsilon$. For a finite word $\omega=x_1x_2\cdots x_n$, the length of $\omega$ is equal to $n$ and denoted by $|\omega|$. We denote by $|\omega|_a$ (resp. $|\omega|_b$) the number of letters of $a$ (resp. $b$) appearing in $\omega$.

Factor property have been studied extensively, such as Lothaire\cite{L1983,L2002}. As a typical sequence over a binary alphabet, the Fibonacci sequence, having many remarkable properties, appears in many aspects of mathematics and computer science, symbolic dynamics, theoretical computer science etc., we refer to Allouche and Shallit\cite{AS2003}. Specifically,
in combinatorial on words, see Berstel\cite{B1966,B1980} for a survey.

As usual, let $\sigma:\mathcal{A}\rightarrow \mathcal{A}^\ast$ be a morphism defined by $\sigma(a)=ab$, $\sigma(b)=a$. As we know, $\mathcal{A}^\ast$ is the free monoid on $\mathcal{A}$, so $\sigma(ab)=\sigma(a)\sigma(b)$. We define the $k$-th iteration of $\sigma$ by $\sigma^k(a)=\sigma^{k-1}(\sigma(a))$, $k\geq2$ and we denote $F_k=\sigma^k(a)$, by convention, we define $\sigma^0(a)=a$ and $\sigma^0(b)=b$.
Then the Fibonacci sequence $F_\infty$ is defined by
$$F_\infty=\lim_{k\rightarrow\infty}F_k=abaababaabaababaababa\cdots,$$
for the details of the properties of the sequence, see\cite{WW1994}.

Let $f_k$ be the $k$-th Fibonacci number, i.e., $f_{-1}=1$, $f_0=1$, $f_1=2$, $f_{k+1}=f_k+f_{k-1}$ for $k\geq 0$. Obviously, $|F_k|=f_k$.

For a finite word $\omega=x_1x_2\cdots x_n$, let $0\leq i\leq n-1$, the word $C_i(\omega):=x_{i+1}\cdots x_nx_1\cdots x_i$ is called the $i$-th conjugation of $\omega$.
If $j\geq n$, $0\leq i\leq n-1$ and $i\equiv j(mod\ n)$, we define $C_j(\omega):=C_i(\omega)$.

Let $\tau=x_1x_2\cdots x_m\cdots$ be a sequence, for any $i\leq j$, define $\tau[i,j]:=x_ix_{i+1}\cdots x_{j-1}x_j$, the factor of $\tau$ of length $j-i+1$, starting from the $i$-th letter and ending to the $j$-th letter. By convention, we note $\tau[i]:=\tau[i,i]=x_i$ and $\tau[i,i-1]:=\varepsilon$.

The notation $\nu\prec\omega$ means that the word $\nu$ is a factor of word $\omega$.

For each fixed $k\geq1$, we always denote by $\alpha\in\{a,b\}$ the last letter of $F_k$. It's easy to see that when $k$ is even then $\alpha=a$ and when $k$ is odd then $\alpha=b$.
Since $F_{k+1}=F_kF_{k-1}$, the last letter of $F_k$ and $F_{k+2j}$ are coincident.
Since $F_0=a$ and $F_1=ab$, then last letter of $F_k$ and $F_{k+2j-1}$ are different,
which we denote by $\beta\in\{a,b\}$. That means $\alpha$ is the last letter of $F_{k+2j}$ and
$\beta$ is the last letter of $F_{k+2j-1}$, for each $j=1,2,\ldots$.
Throughout the paper, we always suppose $\alpha$ and $\beta$ are different letter when
they appear simultaneously.

Let $\nu=\nu_1\nu_2\cdots\nu_n\in\mathcal{A}^\ast$, we denote by $\nu^{-1}:=\nu_n^{-1}\cdots\nu_2^{-1}\nu_1^{-1}$, called the inverse word of $\nu$.
Let $\omega=u\nu$, then $\omega^{-1}=(u\nu)^{-1}=\nu^{-1}u^{-1}$.
Furthermore, $\omega\nu^{-1}=u\nu\nu^{-1}=u$ and $u^{-1}\omega=u^{-1}u\nu=\nu$.

In 1994, Wen Zhi-Xiong and Wen Zhi-Ying\cite{WW1994} introduced two concepts called singular word and singular decomposition. The singular word of order $k$, which we denote by $s_k=\beta F_k\alpha^{-1}$, where $\omega
\alpha^{-1}=\omega[1,|\omega|-1]$, i.e., delete the last letter of $\omega$.
Obviously, we must make sure the last letter of $\omega$ is $\alpha$ before we use notation $\omega\alpha^{-1}$ to express a factor.

The singular word will play an important role in this paper, for more details of singular words and applications, see Cao and Wen\cite{CW2003},
Tan and Wen\cite{TW2007}, which generalized the singular word to Sturmian sequences and Tribonacci sequence respectively.
The singular word has some applications in some aspects, such as Lyndon words\cite{M1999, S2014}, palindromes\cite{G2006}, smooth words\cite{BBC2005}, location of factors\cite{CH2005, CH2010}, Pad$\acute{e}$ approximation\cite{KTW1999, T1999}, etc.

\vspace{0.4cm}

\noindent\emph{1.1~Some Definitions}

\vspace{0.4cm}

Let $\omega$ be factor of $F_\infty$, we will introduce some definitions: factor sequence $\{\omega_p\}_{p\ge 1}$, gap word $\nu_{\omega,p}$, gap sequence $\{\nu_{\omega,p}\}_{p\ge 1}$, and singular kernel $sk(\omega)$.
Especially, when $\omega$ is $s_k$, the singular word of order $k$, the factor sequence $\{s_{k,p}\}_{p\ge 1}$, the gap word $\nu_{s_k,p}$ and gap sequence $\{\nu_{s_k,p}\}_{p\ge 1}$.

\begin{definition}[Factor sequence]
Let $\omega$ be a factor of Fibonacci sequence, then it appears in the sequence infinitely many times, which we arrange by the sequence
$\{\omega_p\}_{p\ge 1}$, where $\omega_p$ denote the $p$-th appearance of $\omega$.
Especially, when $\omega=s_k$, the factor sequence is $\{s_{k,p}\}_{p\ge 1}$,
where $s_{k,p}$ is the $p$-th appearance of $s_k$.

\end{definition}

\noindent\textbf{Remark.}
In this paper, we always use letter "$k$" to express the "order", such as $F_k=\sigma^k(a)$, $f_k=|F_k|$, $s_k=\beta F_k\alpha^{-1}$ etc, use letter "$p$" to express "the $p$-th appearance" of a word in Fibonacci sequence, such as $\omega_p$, $s_{k,p}$ etc, use letter "$n$" to express "the length" of a factor, such as $|\omega|=n$.

\begin{definition}[Gap]
Let $\omega_p=x_{i+1}\cdots x_{i+n}$, $\omega_{p+1}=x_{j+1}\cdots x_{j+n}$, the gap
between $\omega_p$ and $\omega_{p+1}$, denoted by $\nu_{\omega,p}$, is defined by
\begin{equation*}
\nu_{\omega,p}=
\begin{cases}
\varepsilon&when~i+n=j,~\omega_p~and~\omega_{p+1}~are~adjacent;\\
x_{i+n+1}\cdots x_{j}&when~i+n<j,~\omega_p~and~\omega_{p+1}~are~separated;\\
(x_{j+1}\cdots x_{i+n})^{-1}&when~i+n>j,~\omega_p~and~\omega_{p+1}~are~overlapped.
\end{cases}
\end{equation*}
Especially, when $\omega=s_k$, the gap
between $s_{k,p}$ and $s_{k,p+1}$ denoted by $\nu_{s_k,p}$.

The set of gaps of factor $\omega$ is defined as $\{\nu_{\omega,p}|~p\geq1\}$.
\end{definition}

\noindent\textbf{Remark.}
Intuitively when $\omega_p$ and $\omega_{p+1}$ are overlapped,
the overlapped part is the word $x_{j+1}\cdots x_{i+n}$, we take its inverse word as the gap $\nu_{\omega,p}$.
By this way, it is clear to distinguish the cases "adjacent", "separated" and "overlapped".

\vspace{0.2cm}

\noindent\textbf{Example.}
Let $\omega_1=aababaab$, $\omega_2=baabaabab$, then:

(1) $\nu_{\omega_1,1}=\varepsilon$ (adjacent) and $\nu_{\omega_1,2}=(aab)^{-1}$ (overlapped);

(2) $\nu_{\omega_2,1}=aaba$ (separated) and $\nu_{\omega_2,2}=b^{-1}$ (overlapped).

\vspace{0.2cm}

\noindent\textbf{Remark.}
A closed related concept of "Gap" is "Return Word" was introduced by F.Durand\cite{D1998}, for characterizing a sequence over a finite alphabet to be substitutive, he proved that a sequence is primitive substitutive if and only if the set of its return words is finite, which means, each factor of this sequence has finite return words.
In 2001, L.Vuillon\cite{V2001} proved that an infinite word $\tau$ is a Sturmian sequence if and only if each non-empty factor $\omega\prec\tau$ has exactly two distinct return words. Some other related researches(see also \cite{AB2005,BPS2008}) were interested in the cardinality of the set of return words of $\omega$ and the consequent results, but didn't concern about the structures of the sequence derived by return words. Essentially gap words can be derived from the return words which differ from only one prefix $\omega$, but since the terminology "gap" will be convenient and have some advantages for our discussions, we prefer adopt it.

In the present paper, we are interested in the structures of the gap words and the gap sequence, i.e., we will determine the gap words for each factor $\omega$ (in general, the gaps associated with $\omega$ are distinct for different $\omega$), and the structure of the sequence induced by the gap words, we will prove that the sequence over a new alphabet is still a Fibonacci sequence. Moreover, using the structure, we discuss some typical combinatoric properties of the factors.

\vspace{0.2cm}

\begin{definition}[Gap sequence]
Let $\nu_{\omega,p}$ be the gap between $\omega_p$ and $\omega_{p+1}$, we call $\{\nu_{\omega,p}\}_{p\ge 1}$ the gap sequence of the factor $\omega$.
\end{definition}

We will see below that the sequence  $\nu_{\omega,p}$ consists of only two different factors
which we will determine them explicitly, moreover the sequence is still a Fibonacci sequence over a binary alphabet.

\vspace{0.2cm}

Now we are going to introduce "singular kernel" (Definition 1.4) and give "singular decomposition" (Proposition 1.5) which will play an important role in our studies.

\begin{definition}[Singular kernel]
For each $\omega\prec F_\infty$, we denote the longest singular word $s_k$ in $\omega$ by $sk(\omega)$, called the singular kernel of $\omega$.
\end{definition}

We also call "singular kernel" as "kernel" sometimes for short.

\vspace{0.2cm}

\noindent\textbf{Example.}
$sk(baabaa)=aabaa=s_3$, $sk(aababa)=bab=s_2$.

\begin{proposition}[Uniqueness of Singular Kernel and Singular Decomposition]\

Assume that $\omega\prec F_\infty$ and $\omega\not\in\{\varepsilon,ab,ba,aba\}$. Then

(1) $\omega$ has a unique singular kernel $sk(\omega)$, i.e., as a factor, $sk(\omega)$ appears in $\omega$ only once;

(2) $\omega$ has a unique singular decomposition by its singular as $\omega=\mu_1(\omega)\ast sk(\omega)\ast\mu_2(\omega)$.
\end{proposition}

\begin{proof}
Notice that $s_{-1}=b$ and $s_{0}=a$, so if $\omega\prec F_\infty$ with $\omega\neq\varepsilon$, $\omega$ contains a singular word.
When $\omega\not\in\{\varepsilon,ab,ba,aba\}$, there is a singular word $s_i\prec\omega$ with
$i\geq1$. Since $|s_k|=f_k$, we see that if $i<j$, $i\neq-1$, $j\neq0$, $|s_i|<|s_j|$.
Thus, there is a singular word $s_k\prec\omega$, such that if $s_m\prec\omega$ is a singular words and $m\neq k$, then $|s_m|<|s_k|$.
By definition $s_k$ is a singular kernel of $\omega$.
Suppose $s_k$ appears in $\omega$ twice, then by Proposition 1.7(2), either $s_ks_{k+1}s_k$ or $s_ks_{k-1}s_{k}=s_{k+2}$ will be a factor
of $\omega$. In both cases, since $|s_{k+1}|,|s_{k+2}|>|s_k|$, which contracts the hypotheses of $s_k$.
So $s_k$ appears in $\omega$ only once, that is $\omega$ has a unique singular kernel $s_k$.

From the discussion above, we get immediately the conclusion (2).
\end{proof}

\vspace{0.2cm}

\noindent\textbf{Remark.} In this article, we only consider the factor which has a unique singular decomposition by its singular kernel, i.e., $\omega\not\in\{\varepsilon,ab,ba,aba\}$. When $\omega\in\{\varepsilon,ab,ba,aba\}$, their gaps and gap sequences are ready to determine.

\vspace{0.2cm}

The decomposition of Proposition 1.5 will be essential in our research, we will determine the decomposition for any factor, it is not evident to
get the expressions of $sk(\omega)$, $\mu_1(\omega)$ and $\mu_2(\omega)$ in general. To do this, let $\omega$ be a factor with $|\omega|=n$,
since we only know the length of $\omega$, we should first determine all possible order of its singular kernel,
secondly we should study all possible factors neighbor to the kernel. By a carefully analysis, we can classify all factors of $F_\infty$ into
the following six set, called types. In each types, we have an explicit expression of the decomposition.

\begin{definition}[Types]
The sets \emph{Ti,j}, where $i=1,2$, $j=1,2,3$ are defined as follows:

\emph{T1.1}:=$\{\omega\in F_\infty|~|\omega|=f_k, sk(\omega)=s_k\}$;

\emph{T1.2}:=$\{\omega\in F_\infty|~|\omega|=f_k, sk(\omega)=s_{k-1}\}$;

\emph{T1.3}:=$\{\omega\in F_\infty|~|\omega|=f_k, sk(\omega)=s_{k-2}\}$;

\emph{T2.1}:=$\{\omega\in F_\infty|~f_k<|\omega|<f_{k+1}, sk(\omega)=s_k\}$;

\emph{T2.2}:=$\{\omega\in F_\infty|~f_k<|\omega|<f_{k+1}, sk(\omega)=s_{k-1}\}$;

\emph{T2.3}:=$\{\omega\in F_\infty|~f_k<|\omega|<f_{k+1}, sk(\omega)=s_{k-2}\}$.
\end{definition}

In Lemma 4.1 we will prove these six types are pairwise disjoint and their union is all factors of $F_\infty$ with length $f_k\leq|\omega|<f_{k+1}$.

\vspace{0.4cm}

\noindent\emph{1.2 Decompositions of $F_\infty$}

\vspace{0.4cm}

Wen and Wen\cite{WW1994} gave two decompositions of $F_\infty$ as below, and the second one is called the positively separate property of the singular words.

\begin{proposition}[Wen and Wen\cite{WW1994}]\

\vspace{-0.5cm}
\begin{equation*}
\begin{split}(1)~ F_\infty=\prod^\infty_{k=-1}s_k=\underbrace{a}_{s_{-1}}\underbrace{b}_{s_0}\underbrace{aa}_{s_1}\underbrace{bab}_{s_2}\underbrace{aabaa}_{s_3}
\underbrace{babaabab}_{s_4}\underbrace{aabaababaabaa}_{s_5},
\end{split}
\end{equation*}
where $s_k=\beta F_k\alpha^{-1}$, the $k\text{-}th$ singular word;

\vspace{-0.5cm}
\begin{equation*}
\begin{split}(2)~ F_\infty=\left(\prod^{k-1}_{j=-1}s_j\right)s_{k,1}\nu_{s_k,1}s_{k,2}\nu_{s_k,2}\cdots s_{k,p}\nu_{s_k,p}\cdots,~For~any~k\geq0,
\end{split}
\end{equation*}
where the set of gaps $\{\nu_{s_k,p},~p\geq1\}$ has only two elements $s_{k+1}$ and $s_{k-1}$. Furthermore, the gap sequence $\{\nu_{s_k,p}\}_{p\geq1}$ is Fibonacci sequence over the alphabet $\{s_{k+1}, s_{k-1}\}$.

\end{proposition}

\noindent\textbf{Example.}
In the decomposition 1.7(2), take $s_2=bab$. Then two distinct gaps of the word $s_2$ are $\nu_{s_2,1}=s_3=aabaa$ and $\nu_{s_2,2}=s_1=aa$.

\vspace{-0.3cm}
\begin{equation*}
\begin{split}
F_\infty=&a|b|aa(bab)\underbrace{aabaa}_{A}(bab)\underbrace{aa}_{B}(bab)
\underbrace{aabaa}_{A}(bab)\underbrace{aabaa}_{A}(bab)\underbrace{aa}_{B}(bab)\underbrace{aabaa}_{A}(bab)\\
&\underbrace{aa}_{B}(bab)\underbrace{aabaa}_{A}(bab)
\underbrace{aabaa}_{A}(bab)\underbrace{aa}_{B}(bab)
\underbrace{aabaa}_{A}(bab)\underbrace{aabaa}_{A}(bab)\underbrace{aa}_{B}\cdots.
\end{split}
\end{equation*}
We see that the gap sequence $\{\nu_{s_2,p}\}_{p\geq1}=ABAABABAABAAB\cdots$ is Fibonacci sequence over the alphabet $\{A,B\}$.

\vspace{0.2cm}

The first aim of this article is to extended the results for singular words to the general words of the Fibonacci sequence, and discuss the structure of gap sequence $\{\nu_{\omega,p}\}_{p\geq1}$. But in the present case, we have no recurrence relation $s_{k+2}=s_ks_{k-1}s_k$ only valid for singular words. One of main ideas for overcoming the difficulty is to introduce the singular kernel and establishes the relation among four different sequences: $\{\omega_p\}$, $\{s_{k,p}\}$, $\{\nu_{\omega,p}\}$ and $\{\nu_{s_k,p}\}$. We also give the expressions of gaps between each $\omega_p$ and $\omega_{p+1}$.
Since different factors has different gaps, we have no general expressions for them, by carefully observing, we are led to classify some types of the factors by their
characterization, and study respectively the gaps by these types(see Definition 1.6).

The second aim of this article consists of studying some global combinatoric properties of factors, that is, the property depends
on the location of the factor. We will introduce the spectrum $\{(\omega,p)\}$ of a property, that is, $\omega\prec F_\infty$, $p\in\mathbb{N}$, s.t. $\omega_p$ possess the property.
The spectrum $\{(\omega,p)\}$ describes two independent variables $\omega$ and $p$.
Using the results above,  we determine completely the spectrum for some typical combinatoric properties.

\vspace{0.4cm}

\noindent\emph{1.3~Organization of the paper}

\vspace{0.4cm}

The paper is organized as follows.
In Section 2, we state the main results of the paper and give some examples. Section 3 to Section 5 are devoted to the proofs of our main results.
In Section 6, we will define and determine the spectrums of some combinatorial properties of factors.

\vspace{0.5cm}

\stepcounter{section}

\noindent\textbf{\large{2.~Main Results And Examples}}

\vspace{0.4cm}

In this section, we state the main results of this paper and give some simple examples.
The main conclusions are Theorem 2.2 and 2.4 which characterize the structure of the gap sequence: Theorem 2.2 show that there are exactly two distinct elements
and they constitutes still a Fibonacci sequence over a new alphabet; and Theorem 2.4, for any factor $\omega\in F_{\infty}$, gives explicitly the expressions for all gaps.

\vspace{0.2cm}

\noindent\textbf{Proposition 1.5} (Uniqueness of Singular Kernel and Singular Decomposition).

\emph{Assume that $\omega\prec F_\infty$ and $\omega\not\in\{\varepsilon,ab,ba,aba\}$. Then}

\emph{(1) $\omega$ has a unique singular kernel $sk(\omega)$, i.e., as a factor, $sk(\omega)$ appears in $\omega$ only once;}

\emph{(2) $\omega$ has a unique singular decomposition by its singular as $\omega=\mu_1(\omega)\ast sk(\omega)\ast\mu_2(\omega)$.}

\vspace{0.2cm}

Proposition 1.5 states a factor $\omega$ can be decomposed by its singular kernel, and two questions arise naturally: (1) for any factor $\omega$, determine explicitly its decomposition by singular kernel; (2) the sequences $\{\omega_p\}_{p\geq0}$ and $\{s_{k,p}\}_{p\geq0}$ describes the locations of these factors, what is the relation between the singular kernel of $sk(\omega_p)$ and the singular word $s_{k,p}$? Theorem 2.3 answers completely the first question, and Theorem 2.1 answers the second question positively, it shows that $sk(\omega_p)=s_{k,p}$.

\begin{Theorem}[Decomposition of $\omega_p$ and $\nu_{\omega,p}$ by $sk(\omega)$]\

Let $\omega_p\prec F_\infty$ and  $sk(\omega_p)=s_k$. Both decompositions below are unique:

(1) $\omega_p=\mu_1(\omega)\ast s_{k,p}\ast\mu_2(\omega)$;

(2) $\nu_{\omega,p}=\mu^{-1}_2(\omega)\ast\nu_{s_k,p}\ast\mu^{-1}_1(\omega)$.
\end{Theorem}

Using Theorem 2.1, we can prove Theorem 2.2. Furthermore, by Theorem 2.1, we know the key to determine the expressions of gaps is to find out the expressions of $\mu_1$, $\mu_2$ and $sk(\omega)$ for each $\omega$. Theorem 2.3 give us the expressions of them, where we divide the factors into six types. The expressions of gaps are given in Theorem 2.4.

\begin{Theorem}[Gap and gap sequence]\

(1) Any factor $\omega\prec F_{\infty}$ has exactly two distinct gaps $\nu_{\omega,1}$ and $\nu_{\omega,2}$;

(2) The gap sequence $\{\nu_{\omega,p}\}_{p\ge 1}$ is the Fibonacci sequence.
\end{Theorem}

\noindent\textbf{Remark.}
Theorem 2.2(1) has been proved by Vuillon\cite{V2001}, we will give another simple proof in this case.

\vspace{0.2cm}

\noindent\textbf{Example.}
The factor $\omega=abaa$ has exact two gaps $\nu_{\omega,1}=b=:A$, $\nu_{\omega,2}=a^{-1}=:B$. Then
the sequence $\{\nu_{\omega,p}\}_{p\ge 1}=ABAABABA\cdots$ is the Fibonacci sequence over the alphabet $\{A,B\}$.
\begin{equation*}
\begin{split}
F_\infty=&(abaa)\underbrace{b}_A(aba\underbrace{(a)}_B baa)\underbrace{b}_A(abaa)\underbrace{b}_A(aba\underbrace{(a)}_B baa)\underbrace{b}_A(aba\underbrace{(a)}_B baa)\underbrace{b}_A\\
&(abaa)\underbrace{b}_A(aba\underbrace{(a)}_B baa)\underbrace{b}_A(abaa)\underbrace{b}_A(aba\underbrace{(a)}_B baa)\cdots
\end{split}
\end{equation*}

Given a factor $\omega$, by Proposition 1.5, $\omega$ has a unique decomposition by its singular kernel
$\omega=\mu_1(\omega)sk(\omega)\mu_2(\omega)$. For giving the explicitly expression of $\mu_1(\omega)$ and $\mu_2(\omega)$,
we need to divide all factors of $F_\infty$ in to six disjoined types $Ti,j(1\le i\le2, 1\le j\le 3)$(see Definition 1.6).

\begin{Theorem}[Decomposition of $\omega$ by $sk(\omega)$, more explicitly]\

\vspace{-0.6cm}
\begin{equation*}
\omega\in\emph{T1.1}\Rightarrow \mu_1(\omega)=\mu_2(\omega)=\varepsilon.
\end{equation*}

\vspace{-0.6cm}
\begin{equation*}
\omega\in\emph{T1.2}\Rightarrow\begin{cases}
\mu_1(\omega)=s_{k-2}[i-f_{k-1}+2,f_{k-2}],\\
\mu_2(\omega)=s_{k-2}[1,i-f_{k-1}+1],~~f_{k-1}-1\leq i\leq f_k-1.
\end{cases}
\end{equation*}

\vspace{-0.6cm}
\begin{equation*}
\omega\in\emph{T1.3}\Rightarrow\begin{cases}
\mu_1(\omega)=s_{k-1}[i+2,f_{k-1}],\\
\mu_2(\omega)=s_{k-1}[1,i+1],~~0\leq i\leq f_{k-1}-2.
\end{cases}
\end{equation*}

\vspace{-0.6cm}
\begin{equation*}
\omega\in\emph{T2.1}\Rightarrow\begin{cases}
\mu_1(\omega)=s_{k-1}[f_{k-1}-i+1,f_{k-1}],\\
\mu_2(\omega)=s_{k-1}[1,n-f_k-i],~~0\leq i\leq n-f_k.
\end{cases}
\end{equation*}

\vspace{-0.6cm}
\begin{equation*}
\omega\in\emph{T2.2}\Rightarrow\begin{cases}
\mu_1(\omega)=s_k[f_k-i+1,f_k],\\
\mu_2(\omega)=s_k[1,n-f_{k-1}-i],~~0\leq i\leq n-f_{k-1}.
\end{cases}
\end{equation*}

\vspace{-0.6cm}
\begin{equation*}
\omega\in\emph{T2.3}\Rightarrow\begin{cases}
\mu_1(\omega)=s_{k-1}[f_{k+1}-n-i,f_{k-1}],\\
\mu_2(\omega)=s_{k-1}[1,f_{k-1}-i-1],~~0\leq i\leq f_{k+1}-n-2.
\end{cases}
\end{equation*}
\end{Theorem}

\vspace{0.2cm}

\noindent\textbf{Example.}

(1) Take $\omega=babaabab$, $|\omega|=8$, $sk(\omega)=babaabab=s_k$, so $\omega\in$T1.1, $\mu_1=\mu_2=\varepsilon$.

(2) Take $\omega=baabaaba$, $|\omega|=8$, $sk(\omega)=aabaa=s_{k-1}$, so $\omega\in$T1.2, $\mu_1=s_2[3,3]=b$ and $\mu_2=s_2[1,2]=ba$.
The position of $\omega$ in $s_{4-2}s_{4-1}s_{4-2}$ is shown as $ba\underline{b|aabaa|ba}b$.

(3) Take $\omega=baababaa$, $|\omega|=8$, $sk(\omega)=bab=s_{k-2}$, so $\omega\in$T1.3, $\mu_1=s_3[3,5]=baa$ and $\mu_2=s_3[1,2]=aa$.
The position of $\omega$ in $a^{-1}s_{4-1}s_{4-2}s_{4-1}a^{-1}$ is shown as $a\underline{baa|bab|aa}ba$.

(4) Take $\omega=ababaabab$, $|\omega|=9$, $sk(\omega)=babaabab=s_k$, so $\omega\in$T2.1, $\mu_1=s_3[5,5]=a$ and $\mu_2=s_3[1,0]=\varepsilon$.
The position of $\omega$ in $s_{4-1}s_{4}s_{4-1}$ is shown as $aaba\underline{a|babaabab|}aabaa$.

(5) Take $\omega=abaabaaba$, $|\omega|=9$, $sk(\omega)=aabaa=s_{k-1}$, so $\omega\in$T2.2, $\mu_1=s_4[7,8]=ab$ and $\mu_2=s_4[1,2]=ba$.
The position of $\omega$ in $s_{4}s_{4-1}s_{4}$ is shown as $babaab\underline{ab|aabaa|ba}baabab$.

(6) Take $\omega=baababaab$, $|\omega|=9$, $sk(\omega)=bab=s_{k-2}$, so $\omega\in$T2.3, $\mu_1=s_3[3,5]=baa$ and $\mu_2=s_3[1,3]=aab$.
The position of $\omega$ in $a^{-1}s_{4-1}s_{4-2}s_{4-1}a^{-1}$ is shown as $a\underline{baa|bab|aab}a$.

\begin{Theorem}[Expressions of $\nu_{\omega,1}$ and $\nu_{\omega,2}$]\

\vspace{-0.6cm}
\begin{equation*}
\omega\in\emph{T1.1}\Rightarrow\begin{cases}
\nu_{\omega,1}=s_{k+1},&|\nu_{\omega,1}|=f_{k+1}>0,\\
\nu_{\omega,2}=s_{k-1},&|\nu_{\omega,2}|=f_{k-1}>0.
\end{cases}
\end{equation*}

\vspace{-0.6cm}
\begin{equation*}
\omega\in\emph{T1.2}\Rightarrow\begin{cases}
\nu_{\omega,1}=C_{i-f_{k-1}}(F_{k-1}),&|\nu_{\omega,1}|=f_{k-1}>0,\\
\nu_{\omega,2}=\varepsilon,&|\nu_{\omega,2}|=0.
\end{cases}
\end{equation*}

\vspace{-0.6cm}
\begin{equation*}
\omega\in\emph{T1.3}\Rightarrow\begin{cases}
\nu_{\omega,1}=\varepsilon,&|\nu_{\omega,1}|=0,\\
\nu_{\omega,2}^{-1}=C_i(F_{k-2}),&|\nu_{\omega,2}|=-f_{k-2}<0.
\end{cases}
\end{equation*}

\vspace{-0.6cm}
\begin{equation*}
\omega\in\emph{T2.1}\Rightarrow\begin{cases}
\nu_{\omega,1}=s_{k+1}[n-f_k-i+1,f_{k+1}-i],&|\nu_{\omega,1}|=f_{k+2}-n>0,\\
\nu_{\omega,2}=s_{k-1}[n-f_k-i+1,f_{k-1}-i],&|\nu_{\omega,2}|=f_{k+1}-n>0.
\end{cases}
\end{equation*}

\vspace{-0.6cm}
\begin{equation*}
\omega\in\emph{T2.2}\Rightarrow\begin{cases}
\nu_{\omega,1}=s_k[n-f_{k-1}-i+1,f_k-i],&|\nu_{\omega,1}|=f_{k+1}-n>0,\\
\nu_{\omega,2}^{-1}=F_{k+1}[f_k-i,n-i-1],&|\nu_{\omega,2}|=f_k-n<0.
\end{cases}
\end{equation*}

\vspace{-0.6cm}
\begin{equation*}
\omega\in\emph{T2.3}\Rightarrow\begin{cases}
\nu_{\omega,1}^{-1}=s_{k-1}[f_{k+1}-n-i,f_{k-1}-i-1],&|\nu_{\omega,1}|=f_k-n<0,\\
\nu_{\omega,2}^{-1}=F_k[f_{k+1}-n-i-1,f_k-i-2],&|\nu_{\omega,2}|=f_{k-1}-n<0.
\end{cases}
\end{equation*}
\end{Theorem}

\noindent\textbf{Example.}

(1) Take $\omega=babaabab\in$T1.1, then $n=8$, $k=4$.

\vspace{-0.6cm}
\begin{equation*}
\Rightarrow\begin{cases}
\nu_{\omega,1}=s_5=aabaababaabaa,&|\nu_{\omega,1}|=f_5=13>0,\\
\nu_{\omega,2}=s_3=aabaa,&|\nu_{\omega,2}|=f_3=5>0.
\end{cases}
\end{equation*}
\vspace{-0.4cm}

(2) Take $\omega=baabaaba\in$T1.2, then $n=8$, $k=4$, $i=6$.

\vspace{-0.6cm}
\begin{equation*}
\Rightarrow\begin{cases}
\nu_{\omega,1}=C_1(F_3)=baaba,&|\nu_{\omega,1}|=f_3=5>0,\\
\nu_{\omega,2}=\varepsilon,&|\nu_{\omega,2}|=0.
\end{cases}
\end{equation*}
\vspace{-0.4cm}

(3) Take $\omega=baababaa\in$T1.3, then $n=8$, $k=4$, $i=1$.

\vspace{-0.6cm}
\begin{equation*}
\Rightarrow\begin{cases}
\nu_{\omega,1}=\varepsilon,&|\nu_{\omega,1}|=0,\\
\nu_{\omega,2}^{-1}=C_1(F_2)=baa,&|\nu_{\omega,2}|=-f_2=-3<0.
\end{cases}
\end{equation*}
\vspace{-0.4cm}

(4) Take $\omega=ababaabab\in$T2.1, then $n=9$, $k=4$, $i=1$.

\vspace{-0.6cm}
\begin{equation*}
\Rightarrow\begin{cases}
\nu_{\omega,1}=s_5[1,12]=aabaababaaba,&|\nu_{\omega,1}|=f_6-9=12>0,\\
\nu_{\omega,2}=s_3[1,4]=aaba,&|\nu_{\omega,2}|=f_5-9=4>0.
\end{cases}
\end{equation*}
\vspace{-0.4cm}

(5) Take $\omega=abaabaaba\in$T2.2, then $n=9$, $k=4$, $i=2$.

\vspace{-0.6cm}
\begin{equation*}
\Rightarrow\begin{cases}
\nu_{\omega,1}=s_4[3,6]=baab,&|\nu_{\omega,1}|=f_5-9=4>0,\\
\nu_{\omega,2}^{-1}=F_5[6,6]=a,&|\nu_{\omega,2}|=f_4-9=-1<0.
\end{cases}
\end{equation*}
\vspace{-0.4cm}

(6) Take $\omega=baababaab\in$T2.3, then $n=9$, $k=4$, $i=1$.

\vspace{-0.6cm}
\begin{equation*}
\Rightarrow\begin{cases}
\nu_{\omega,1}^{-1}=s_3[3,3]=b,&|\nu_{\omega,1}|=f_4-9=-1<0,\\
\nu_{\omega,2}^{-1}=F_4[2,5]=baab,&|\nu_{\omega,2}|=f_3-9=-4<0.
\end{cases}
\end{equation*}


\vspace{0.5cm}

\stepcounter{section}

\noindent\textbf{\large{3.~Proofs of Theorem 2.1 and Theorem 2.2}}

\vspace{0.4cm}

In this section, we will prove Theorem 2.1 and Theorem 2.2.

Theorem 2.1 establishes the relations among four sequences below:

(1) Factor sequence $\{\omega_p\}_{p\geq1}$, $\omega_p$ is the $p$-th appearance of $\omega$;

(2) Factor sequence $\{s_{k,p}\}_{p\geq1}$, $s_k$ is the singular kernel of $\omega$;

(3) Gap sequence $\{\nu_{\omega,p}\}_{p\geq1}$, $\nu_{\omega,p}$ is the gap between $\omega_p$ and $\omega_{p+1}$;

(4) Gap sequence $\{\nu_{s_k,p}\}_{p\geq1}$, $\nu_{s_k,p}$ is the gap between $s_{k,p}$ and $s_{k,p+1}$.

As one of our main conclusion, Theorem 2.2 shows that there are exactly two distinct elements and they constitutes still a Fibonacci sequence over a new alphabet.

We give some lemmas first which are very useful in the studies.

\begin{lemma} Let $s_k$ be the $k$-th singular word. Then:

(1) $s_k=\beta\alpha^{-1}s_{k-1}s_{k-2}=s_{k-2}s_{k-1}\alpha^{-1}\beta;$

(2) $s_k=s_{k-2}s_{k-3}s_{k-2}.$
\end{lemma}

\begin{proof}
(1) As we declared in Section 1, $\alpha$ is the last letter of $F_k$ and $F_{k-2}$, $\beta$ is the last letter of $F_{k-1}$, $\beta\neq\alpha$.
Since $F_k=F_{k-1}F_{k-2}$, we know $\beta F_k\alpha^{-1}=\beta\alpha^{-1}\alpha F_{k-1}\beta^{-1}\beta F_{k-2}\alpha^{-1}$. By the definition of singular word, $s_k=\beta\alpha^{-1}s_{k-1}s_{k-2}$, we thus get the first equality in (1).

By induction, we can prove $F_k=F_{k-2}F_{k-1}\beta^{-1}\alpha^{-1}\beta\alpha$, so
$$\beta F_k\alpha^{-1}=\beta F_{k-2}F_{k-1}\beta^{-1}\alpha^{-1}\beta\alpha\alpha^{-1}=\beta F_{k-2}\alpha^{-1}\alpha F_{k-1}\beta^{-1}\alpha^{-1}\beta,$$
which yields that $s_k=s_{k-2}s_{k-1}\alpha^{-1}\beta$, and concludes the second equality in (1).

(2) By (1), $s_k=\beta\alpha^{-1}s_{k-1}s_{k-2}$ and $s_{k-1}=\alpha\beta^{-1}s_{k-2}s_{k-3}$, so $$s_k=\beta\alpha^{-1}\ast\alpha\beta^{-1}s_{k-2}s_{k-3}\ast s_{k-2}=s_{k-2}s_{k-3}s_{k-2}.$$
\end{proof}

\begin{lemma}
$\prod^{k-1}_{j=-1}s_j=\alpha^{-1}s_{k+1}$.
\end{lemma}

\begin{proof} By induction.
(1) When $k=0$, $\prod^{-1}_{j=-1}s_j=s_{-1}=a$ and $\alpha^{-1}s_{1}=a^{-1}aa=a$, where $\alpha$ is the last letter of $F_k$. When $k=1$, $\prod^{0}_{j=-1}s_j=s_{-1}s_0=ab$ and $\alpha^{-1}s_{2}=b^{-1}bab=ab$. The proposition holds.

(2) Assume the conclusion holds for $k-1$, $\prod^{k-1}_{j=-1}s_j=\alpha^{-1}s_{k+1}$, then:
\begin{equation*}
\begin{split}
&\prod^{k}_{j=-1}s_j=\left(\prod^{k-1}_{j=-1}s_j\right)s_k=\alpha^{-1}s_{k+1}s_k=\alpha^{-1}\alpha F_{k+1}\beta^{-1}\beta F_k\alpha^{-1}\\
=&F_{k+1}F_k\alpha^{-1}=F_{k+2}\alpha^{-1}=\beta^{-1}s_{k+2},
\end{split}
\end{equation*}
where $\alpha$ is the last letter of $F_k$, $\beta$ is the last letter of $F_{k\pm1}$.
Thus the conclusion holds for $k$, and we prove the proposition.
\end{proof}

\noindent\textbf{Remark.} As a simple corollary of Lemma 3.2, we have $\sum^{k-1}_{j=-1}f_j=f_{k+1}-1$.

\begin{lemma}
For any $\omega$ fixed, let $sk(\omega)=s_k$, then $sk(\omega_p)=s_{k,p}$, i.e., the singular kernel of $\omega_p$ is equal to $s_{k,p}$ by location.

\end{lemma}

\begin{proof}
We will prove the following claims:

Claim (1): For any $p$, there exists $q$, such that $sk(\omega_p)=s_{k,q}$;

Claim (2): For any $q$, there exists $p$, such that $sk(\omega_p)=s_{k,q}$;

Claim (3): If both $s_{k,q_1}$ and $s_{k,q_2}$ are singular kernel of $\omega_p$, then $q_1=q_2$;

Claim (4): If both $sk(\omega_{p_1})$ and $sk(\omega_{p_2})$ are $s_{k,q}$, then $p_1=p_2$.

Since $s_k\prec\omega$, Claim (1) is trivial. By Proposition 1.5, each $\omega$ has a unique decomposition by its singular kernel, so both Claim (3) and (4) are true. It rest to prove Claim (2).

\textbf{Proof of Claim (2).}

Using the positively separate property of the singular word $s_k$ and Lemma 3.2, we get
$$F_\infty=\alpha^{-1}s_{k+1}s_ks_{k+1}s_ks_{k-1}s_ks_{k+1}s_ks_{k+1}s_ks_{k-1}s_k\cdots$$
So for any $q$, the singular words neighboring to $s_{k,q}$ have only five possible cases below.
We use 'underline' to emphasize the singular word $s_{k,q}$ we consider.

Case 1: $\cdots s_ks_{k+1}\underline{s_{k,q}}s_{k-1}s_k\cdots$

Case 2: $\cdots s_ks_{k+1}\underline{s_{k,q}}s_{k+1}s_k\cdots$

Case 3: $\cdots s_ks_{k-1}\underline{s_{k,q}}s_{k+1}s_k\cdots$

Case 4: $\cdots s_ks_{k-1}\underline{s_{k,q}}s_{k-1}s_k\cdots$

Case 5: when $q=1$, $\alpha^{-1}s_{k+1}\underline{s_{k,1}}s_{k+1}s_k\cdots$

We want to prove: there is a $\omega_p$ such that $\omega_p\succ s_{k,q}$ and $sk(\omega_p)=s_{k,q}$ in each case.
Since $sk(\omega)=s_k$, by Proposition 1.5, we only need to find two constant words $\mu_1$ and $\mu_2$ such that $\omega=\mu_1s_k\mu_2$ in all cases.

In case 1, since $sk(\omega)=s_k$, $\omega$ must be factor of $\alpha^{-1}s_{k+1}\underline{s_k}s_{k-1}s_k\beta^{-1}$, i.e., $\alpha^{-1}s_{k+1}\underline{s_k}s_{k+1}\alpha^{-1}$ with kernel $s_k$. Otherwise, $\omega$ contains $s_{k+1}$ or $s_ks_{k-1}s_k=s_{k+2}$, then $sk(\omega)=s_{k+1}$ or $s_{k+2}$, which contradict $sk(\omega)=s_k$. Similarly, in case 2, 3, 4, 5, $\omega$ must be the factor of $\alpha^{-1}s_{k+1}\underline{s_k}s_{k+1}\alpha^{-1}$ with kernel $s_k$ too.
So, in all cases, $\mu_1$ is the suffix of $\alpha^{-1}s_{k+1}$ and $\mu_2$ is the prefix of $s_{k+1}\alpha^{-1}$. Both of them are constant words throughout the five cases.
\end{proof}

By the proof of Lemma 3.3 and Proposition 1.5, we have the corollary below.

\begin{corollary}
Let $s_k$ be the singular word of order $k$, $\theta_k:=\alpha^{-1}s_{k+1}\underline{s_k}s_{k+1}\alpha^{-1}$.

(1) $sk(\theta_k)=s_k$;

(2) If $\tau\prec\theta_k$ with $sk(\tau)=s_k$, then $\tau$ appears in $\theta_k$ only once;

(3) Let $\omega$ be a factor with singular kernel $s_k$, then $\omega\prec\theta_k$, i.e.,
$$\{\omega\prec F_\infty|\ sk(\omega)=s_k\}=\{\omega\prec F_\infty|\ \omega\prec\theta_k,\ sk(\omega)=s_k\}.$$

\end{corollary}

\vspace{0.2cm}

\noindent\textbf{Theorem 2.1}(Decomposition of $\omega_p$ and $\nu_{\omega,p}$ by $sk(\omega)$).

\emph{Let $\omega_p\prec F_\infty$ and $sk(\omega)=s_k$. Both decompositions below are unique:}

\emph{(1) $\omega_p=\mu_1(\omega)\ast s_{k,p}\ast\mu_2(\omega)$;}

\emph{(2) $\nu_{\omega,p}=\mu^{-1}_2(\omega)\ast\nu_{s_k,p}\ast\mu^{-1}_1(\omega)$.}

\begin{proof}
The proof of the proposition will be easy by the following diagram.


\centerline{Fig. 3.1: The relation among $\{\omega_p\}$, $\{s_{k,p}\}$, $\{\nu_{\omega,p}\}$ and $\{\nu_{s_k,p}\}$.}
\end{proof}

\vspace{0.2cm}

\noindent\textbf{Theorem 2.2}(Gap and gap sequence).

\emph{(1) Any factor $\omega\prec F_{\infty}$ has exactly two distinct gaps $\nu_{\omega,1}$ and $\nu_{\omega,2}$;}

\emph{(2) The gap sequence $\{\nu_{\omega,p}\}_{p\ge 1}$ is the Fibonacci sequence.}

\vspace{0.2cm}

\begin{proof}
(1) Since $\nu_{\omega,p}=\mu^{-1}_2(\omega)\ast\nu_{s_k,p}\ast\mu^{-1}_1(\omega)$ and $\{\{\nu_{s_k,p}\}_{p\geq1}\}=\{\nu_{s_k,1},\nu_{s_k,2}\}$, so $\{\{\nu_{\omega,p}\}_{p\geq1}\}=\{\mu^{-1}_2(\omega)\ast\nu_{s_k,1}\ast\mu^{-1}_1(\omega),
\mu^{-1}_2(\omega)\ast\nu_{s_k,2}\ast\mu^{-1}_1(\omega)\}
=:\{\nu_{\omega,1},\nu_{\omega,2}\}$. That means, any factor $\omega$ has exactly two distinct gaps $\nu_{\omega,1}$ and $\nu_{\omega,2}$.

(2) Since $\nu_{\omega,p}=\mu^{-1}_2(\omega)\ast\nu_{s_k,p}\ast\mu^{-1}_1(\omega)$ and the gap sequence $\{\nu_{s_k,p}\}_{p\ge 1}$ is the Fibonacci sequence (see Proposition 1.7(2)), the gap sequence $\{\nu_{\omega,p}\}_{p\ge 1}=\{\mu^{-1}_2(\omega)\ast\nu_{s_k,p}\ast\mu^{-1}_1(\omega)\}_{p\ge 1}$ is the Fibonacci sequence.
\end{proof}


\vspace{0.5cm}

\stepcounter{section}

\noindent\textbf{\large{4.~Proof of Theorem 2.3}}

\vspace{0.4cm}

By Proposition 1.5, we know that each $\omega\prec F_\infty$ has a unique decomposition by its singular kernel: $\omega=\mu_1(\omega)\ast sk(\omega)\ast\mu_2(\omega)$.
We are going to prove Theorem 2.3. which determines the expressions of $sk(\omega)$, $\mu_1(\omega)$ and $\mu_2(\omega)$ for each $\omega$. To solve this problem, we divide the factors into six
types in Definition 1.6. Lemma 4.1 give the relations among these six types, where the notation "$\sqcup$" means pairwise disjoint union.

\begin{lemma} The six types are pairwise disjoint and their union is all factors of $F_\infty$, i.e.,

(1) $\{\omega\in F_\infty|~\exists~k,~s.t.~|\omega|=f_k\}=\emph{T1.1}\sqcup\emph{T1.2}\sqcup\emph{T1.3}$;

(2) $\{\omega\in F_\infty|~\exists~k,~s.t.~f_k<|\omega|<f_{k+1}\}=\emph{T2.1}\sqcup\emph{T2.2}\sqcup\emph{T2.3}$.

\end{lemma}

\begin{proof}

(1) Since $|sk(\omega)|\leq|\omega|$, $sk(\omega)$ can not be singular word $s_j$ with $j>k$. So $sk(\omega)$ can only be $s_j$ with $j\leq k$.
On the other hand, assume $sk(\omega)=s_{k-3}$, then by Corollary 3.4(3), $\omega\prec\beta^{-1}s_{k-2}\underline{s_{k-3}}s_{k-2}\beta^{-1}$ with kernel $s_{k-3}$ . But
$$|\beta^{-1}s_{k-2}s_{k-3}s_{k-2}\beta^{-1}|=2\ast f_{k-2}+f_{k-3}-2=f_k-2<f_k=|\omega|,$$
so $sk(\omega)$ can not be singular word $s_{k-3}$. By an analogous argument, $sk(\omega)$ can not take singular word $s_j$ for $-1\leq j\leq k-4$ too.
That is, $sk(\omega)$ has only three possible cases: $s_k$, $s_{k-1}$, $s_{k-2}$, and by the definitions of T1.1, T1.2 and T1.3, we get
$$\{\omega\in F_\infty|~|\omega|=f_k\}=\textrm{T1.1}\cup\textrm{T1.2}\cup\textrm{T1.3}.$$
Furthermore by Proposition 1.5(1), $\omega$ has a unique singular kernel, so T1.1, T1.2 and T1.3 are pairwise disjoint.

(2) Since $|\omega|<f_{k+1}$ and $|sk(\omega)|\leq|\omega|$, $sk(\omega)$ can not be singular word $s_j$ for $j>k$.
As well as the discussion in the case (1), $sk(\omega)$ can not take singular word $s_j(-1\le j\le k-3)$,
and may take only three cases: $s_k$, $s_{k-1}$, $s_{k-2}$. That is,
$$\{\omega\in F_\infty|~|\omega|=f_k\}=\textrm{T2.1}\cup\textrm{T2.2}\cup\textrm{T2.3}.$$
Moreover, from Proposition 1.5(1), $\omega$ has a unique singular word, so T2.1, T2.2 and T2.3 are pairwise disjoint.
\end{proof}

\vspace{0.2cm}

\noindent\textbf{Theorem 2.3} (Decomposition of $\omega$ by $sk(\omega)$).

\vspace{-0.6cm}
\begin{equation*}
\omega\in\textrm{T1.1}\Rightarrow \mu_1(\omega)=\mu_2(\omega)=\varepsilon.
\end{equation*}

\vspace{-0.6cm}
\begin{equation*}
\omega\in\textrm{T1.2}\Rightarrow\begin{cases}
\mu_1(\omega)=s_{k-2}[i-f_{k-1}+2,f_{k-2}],\\
\mu_2(\omega)=s_{k-2}[1,i-f_{k-1}+1],~~f_{k-1}-1\leq i\leq f_k-1.
\end{cases}
\end{equation*}

\vspace{-0.6cm}
\begin{equation*}
\omega\in\textrm{T1.3}\Rightarrow\begin{cases}
\mu_1(\omega)=s_{k-1}[i+2,f_{k-1}],\\
\mu_2(\omega)=s_{k-1}[1,i+1],~~0\leq i\leq f_{k-1}-2.
\end{cases}
\end{equation*}

\vspace{-0.6cm}
\begin{equation*}
\omega\in\textrm{T2.1}\Rightarrow\begin{cases}
\mu_1(\omega)=s_{k-1}[f_{k-1}-i+1,f_{k-1}],\\
\mu_2(\omega)=s_{k-1}[1,n-f_k-i],~~0\leq i\leq n-f_k.
\end{cases}
\end{equation*}

\vspace{-0.6cm}
\begin{equation*}
\omega\in\textrm{T2.2}\Rightarrow\begin{cases}
\mu_1(\omega)=s_k[f_k-i+1,f_k],\\
\mu_2(\omega)=s_k[1,n-f_{k-1}-i],~~0\leq i\leq n-f_{k-1}.
\end{cases}
\end{equation*}

\vspace{-0.6cm}
\begin{equation*}
\omega\in\textrm{T2.3}\Rightarrow\begin{cases}
\mu_1(\omega)=s_{k-1}[f_{k+1}-n-i,f_{k-1}],\\
\mu_2(\omega)=s_{k-1}[1,f_{k-1}-i-1],~~0\leq i\leq f_{k+1}-n-2.
\end{cases}
\end{equation*}

\begin{proof}
(1) If $\omega\in$T1.1, then $|\omega|=f_k$ and $sk(\omega)=s_k$. Notice that $|s_k|=f_k$, we get $|\mu_1|=|\mu_2|=0$,
i.e., $\mu_1(\omega)=\mu_2(\omega)=\varepsilon$. We have therefore in this case, $\omega=s_k$.

(2) If $\omega\in$T1.2, then $|\omega|=f_k$ and $sk(\omega)=s_{k-1}$. By Corollary 3.4,
$\omega\prec\beta^{-1}s_k\underline{s_{k-1}}s_k\beta^{-1}$, and by Lemma 3.1(2),
$$\beta^{-1}s_k\underline{s_{k-1}}s_k\beta^{-1}=\beta^{-1}s_{k-2}s_{k-3}s_{k-2}\underline{s_{k-1}}s_{k-2}s_{k-3}s_{k-2}\beta^{-1}.$$
Since $|sk(\omega)|=f_{k-1}$, $|\mu_1|+|\mu_2|=f_k-f_{k-1}=f_{k-2}$, which means $\mu_1$ is suffix of $s_{k-2}$ and $\mu_2$ is prefix of $s_{k-2}$.
So $\omega\prec s_{k-2}\underline{s_{k-1}}s_{k-2}$.
We get therefore, $\mu_1=s_{k-2}[i-f_{k-1}+2,f_{k-2}]$, $\mu_2=s_{k-2}[1,i-f_{k-1}+1]$
and $\omega=s_{k-2}[i-f_{k-1}+2,f_{k-2}]s_{k-1}s_{k-2}[1,i-f_{k-1}+1]$, where $f_{k-1}-1\leq i\leq f_k-1$.

(3) If $\omega\in$T1.3, then $|\omega|=f_k$ and $sk(\omega)=s_{k-2}$. By corollary 3.4,
$\omega\prec\alpha^{-1}s_{k-1}\underline{s_{k-2}}s_{k-1}\alpha^{-1}$.
In this case, $\mu_1=s_{k-1}[i+2,f_{k-1}]$, $\mu_2=s_{k-1}[1,i+1]$ and $\omega=s_{k-1}[i+2,f_{k-1}]s_{k-2}s_{k-1}[1,i+1]$, where $0\leq i\leq f_{k-1}-2$.

(4) The conclusion for $\omega\in$T2.1 can be obtained by the same argument as in (2).
The conclusions for $\omega$ being in T2.2 or T2.3  can be obtained by the same argument as in (3).
\end{proof}

\noindent\textbf{Remark.} By Theorem 2.3 and Corollary 3.4(2) the cardinality of each type are:

$\sharp$T1.1$=1$, $\sharp$T1.2$=f_{k-2}+1$, $\sharp$T1.3$=f_{k-1}-1$;

$\sharp$T2.1$=n-f_k+1$, $\sharp$T2.2$=n-f_{k-1}+1$, $\sharp$T2.3$=f_{k+1}-n-1$.

Let $\rho(n)$ is the complexity function of Fibonacci sequence which is defined by the cardinality of the set of the factors with length $n$, then above formulas give
immediately the known result $\rho(n)=n+1$.

\begin{corollary}\

(1) $\omega\in$\emph{T1.2}$\Leftrightarrow\omega=C_i(F_k)$, where $f_{k-1}-1\leq i\leq f_k-1$.

(2) $\omega\in$\emph{T1.3}$\Leftrightarrow\omega=C_i(F_k)$, where $0\leq i\leq f_{k-1}-2$.
\end{corollary}

\begin{proof}
(1) By Theorem 2.3, $\omega\in\text{T1.2}\Leftrightarrow \omega=s_{k-2}[i-f_{k-1}+2,f_{k-2}]s_{k-1}s_{k-2}[1,i-f_{k-1}+1].$
By the definition of conjugate word,
$$\omega=C_{i-f_{k-1}+1}(s_{k-2}s_{k-1})=C_{i-f_{k-1}+1}(s_{k-2}\alpha\alpha^{-1}s_{k-1})=C_i(\alpha^{-1}s_{k-1}s_{k-2}\alpha)=C_i(F_k),$$
where $f_{k-1}-1\leq i\leq f_k-1$.

(2) In this case, $\omega\in\text{T1.3}\Leftrightarrow \omega=s_{k-1}[i+2,f_{k-1}]s_{k-2}s_{k-1}[1,i+1]$, and
$$\omega=s_{k-1}[i+2,f_{k-1}]s_{k-2}\alpha\alpha^{-1}s_{k-1}[1,i+1]=C_i(s_{k-1}[2,f_{k-1}]s_{k-2}s_{k-1}[1,1])=C_i(F_k),$$
where $0\leq i\leq f_{k-1}-2$.
\end{proof}


\vspace{0.5cm}

\stepcounter{section}

\noindent\textbf{\large{5.~Proof of Theorem 2.4}}

\vspace{0.4cm}
This section is devoted to the proof of Theorem 2.4 which gives explicitly the expressions of all gaps for each factors $\omega$.
By Theorem 2.2, we only need to determine the expressions of gaps $\nu_{\omega,1}$ and $\nu_{\omega,2}$.

Suppose $|\omega|=n$ with $f_k\leq n<f_{k+1}$ for some $k$. We divide the proof of Theorem 2.4 into six parts, i.e.,
Theorem 2.4(1) to Theorem 2.4(6) according to the six types in Theorem 2.3.
By Theorem 2.1, $\nu_{\omega,p}=\mu_2^{-1}\nu_{sk(\omega),p}\mu_1^{-1}$ and $\nu^{-1}_{\omega,p}=\mu_1\nu^{-1}_{sk(\omega),p}\mu_2$. By Proposition 1.7(2), $\nu_{s_k,1}=s_{k+1}$ and $\nu_{s_k,2}=s_{k-1}$.

\vspace{0.2cm}

\noindent\textbf{Theorem 2.4(1)}

\vspace{-0.6cm}
\begin{equation*}
\omega\in\textrm{T1.1}\Rightarrow\begin{cases}
\nu_{\omega,1}=s_{k+1},&|\nu_{\omega,1}|=f_{k+1}>0,\\
\nu_{\omega,2}=s_{k-1},&|\nu_{\omega,2}|=f_{k-1}>0.
\end{cases}
\end{equation*}

\begin{proof} Since $\omega$ is in T1.1, $\omega=s_k$ by Theorem 2.3.
So $\nu_{\omega,1}=s_{k+1}$ and $\nu_{\omega,2}=s_{k-1}$.
\end{proof}

\vspace{0.2cm}

\noindent\textbf{Theorem 2.4(2)}

\vspace{-0.6cm}
\begin{equation*}
\omega\in\textrm{T1.2}\Rightarrow\begin{cases}
\nu_{\omega,1}=C_{i-f_{k-1}}(F_{k-1}),&|\nu_{\omega,1}|=f_{k-1}>0,\\
\nu_{\omega,2}=\varepsilon,&|\nu_{\omega,2}|=0.
\end{cases}
\end{equation*}

\begin{proof} Let $\omega\in$T1.2, by Theorem 2.3,
$$\omega=s_{k-2}[i-f_{k-1}+2,f_{k-2}]\underline{s_{k-1}}s_{k-2}[1,i-f_{k-1}+1],\ \ f_{k-1}-1\leq i\leq f_k-1,$$
so $\mu_1=s_{k-2}[i-f_{k-1}+2,f_{k-2}]$ and $\mu_2=s_{k-2}[1,i-f_{k-1}+1]$.

(1) Since $\nu_{s_{k-1},1}=s_k$:
\begin{equation*}
\begin{split}
&\nu_{\omega,1}=\mu_2^{-1}\nu_{s_{k-1},1}\mu_1^{-1}=s^{-1}_{k-2}[1,i-f_{k-1}+1]s_ks_{k-2}^{-1}[i-f_{k-1}+2,f_{k-2}]\\
=&s^{-1}_{k-2}[1,i-f_{k-1}+1](s_{k-2}s_{k-3}s_{k-2})s_{k-2}^{-1}[i-f_{k-1}+2,f_{k-2}]\\
=&s_{k-2}[i-f_{k-1}+2,f_{k-2}]s_{k-3}s_{k-2}[1,i-f_{k-1}+1]=C_{i-f_{k-1}+1}(s_{k-2}s_{k-3})=C_{i-f_{k-1}}(F_{k-1}),
\end{split}
\end{equation*}
which yields that $\nu_{\omega,1}=C_{i-f_{k-1}}(F_{k-1})$, $|\nu_{\omega,1}|=f_{k-1}>0$.

(2) Since $\nu_{s_{k-1},2}=s_{k-2}$:
\begin{equation*}
\begin{split}
&\nu_{\omega,2}=\mu_2^{-1}\nu_{s_{k-1},2}\mu_1^{-1}=s^{-1}_{k-2}[1,i-f_{k-1}+1]s_{k-2}s_{k-2}^{-1}[i-f_{k-1}+2,f_{k-2}]=\varepsilon,
\end{split}
\end{equation*}
so $\nu_{\omega,2}=\varepsilon$, $|\nu_{\omega,2}|=0$.
\end{proof}

\noindent\textbf{Example.} Taking $\omega=baabaaba\in$T1.2, it appears in $F_\infty$ as:

$F_\infty=abaaba(baabaaba)baaba(baabaaba)(baabaaba)baaba(baabaaba)baaba(baabaaba)\cdots$

Theorem 2.4 gives the expressions of gaps, where $n=8$, $k=4$, $i=6$:
\begin{equation*}
\begin{cases}
\nu_{\omega,1}=C_{i-f_{k-1}}(F_{k-1})=C_1(F_3)=baaba,&|\nu_{\omega,1}|=f_3=5>0,\\
\nu_{\omega,2}=\varepsilon,&|\nu_{\omega,2}|=0.
\end{cases}
\end{equation*}

\vspace{0.2cm}

\noindent\textbf{Theorem 2.4(3)}

\vspace{-0.6cm}
\begin{equation*}
\omega\in\textrm{T1.3}\Rightarrow\begin{cases}
\nu_{\omega,1}=\varepsilon,&|\nu_{\omega,1}|=0,\\
\nu_{\omega,2}^{-1}=C_i(F_{k-2}),&|\nu_{\omega,2}|=-f_{k-2}<0.
\end{cases}
\end{equation*}

\begin{proof} Since $\omega$ is in T1.3, by Theorem 2.3,
$$\omega=s_{k-1}[i+2,f_{k-1}]\underline{s_{k-2}}s_{k-1}[1,i+1],$$
where $0\leq i\leq f_{k-1}-2$, $\mu_1=s_{k-1}[i+2,f_{k-1}]$ and $\mu_2=s_{k-1}[1,i+1]$.

(1) Since $\nu_{s_{k-2},1}=s_{k-1}$:
\begin{equation*}
\begin{split}
&\nu_{\omega,1}=\mu_2^{-1}\nu_{s_{k-2},1}\mu_1^{-1}=s^{-1}_{k-1}[1,i+1]s_{k-1}s^{-1}_{k-1}[i+2,f_{k-1}]=\varepsilon,
\end{split}
\end{equation*}
i.e., $\nu_{\omega,1}=\varepsilon$, $|\nu_{\omega,1}|=0$.

(2) Since $\nu_{s_{k-2},1}=s_{k-3}$:
\begin{equation*}
\begin{split}
&\nu^{-1}_{\omega,2}=\mu_1\nu^{-1}_{s_{k-2},2}\mu_2=s_{k-1}[i+2,f_{k-1}]s^{-1}_{k-3}s_{k-1}[1,i+1].
\end{split}
\end{equation*}

We are going to determine the expression of $\nu^{-1}_{\omega,2}$ which we divide into three different cases.

(2.1) If $s_{k-3}$ is the suffix of $s_{k-1}[i+2,f_{k-1}]$, then
$$|s_{k-1}[i+2,f_{k-1}]|=f_{k-1}-i-1\geq f_{k-3},$$
i.e., $i\leq f_{k-2}-1.$
\begin{equation*}
\begin{split}
&\nu^{-1}_{\omega,2}=(s_{k-3}s_{k-4}s_{k-3})[i+2,f_{k-1}]s^{-1}_{k-3}(s_{k-3}s_{k-4}s_{k-3})[1,i+1]\\
=&(s_{k-3}s_{k-4})[i+2,f_{k-1}](s_{k-3}s_{k-4})[1,i+1]=C_{i+1}(s_{k-3}s_{k-4})=C_i(F_{k-2}).
\end{split}
\end{equation*}

(2.2) If $s_{k-3}$ is the prefix of $s_{k-1}[1,i+1]$, then
$$|s_{k-1}[1,i+1]|=i+1\geq f_{k-3},$$
i.e., $i\geq f_{k-3}-1$.
\begin{equation*}
\begin{split}
&\nu^{-1}_{\omega,2}=(s_{k-3}s_{k-4}s_{k-3})[i+2,f_{k-1}]s^{-1}_{k-3}(s_{k-3}s_{k-4}s_{k-3})[1,i+1]\\
=&(s_{k-4}s_{k-3})[i-f_{k-3}+2,f_{k-2}](s_{k-4}s_{k-3})[1,i-f_{k-3}+1]\\
=&C_{i-f_{k-3}+1}(s_{k-4}s_{k-3})=C_i(F_{k-2}).
\end{split}
\end{equation*}

(2.3) If $s_{k-3}$ is neither the suffix of $s_{k-1}[i+2,f_{k-1}]$ nor the prefix of $s_{k-1}[1,i+1]$,
then $f_{k-2}-1<i<f_{k-3}-1$, we have thus $f_{k-2}<f_{k-3}$, which is obviously not true.
So the third case does not exist.

According to (2.1)-(2.3), we have $\nu_{\omega,2}^{-1}=C_i(F_{k-2})$, $|\nu_{\omega,2}|=-f_{k-2}<0$.
\end{proof}

\noindent\textbf{Example.}Taking $\omega=baababaa\in$T1.3, it appears in $F_\infty$ as:

$F_\infty=a(baababaa)(baaba(baa)babaa)(baababaa)(baaba(baa)babaa)(baaba(baa)babaa)(ba\cdots$\\
where $(baa)$ in $(baaba(baa)babaa)$ shows a overlap between two successive $\omega$.

Theorem 2.4 gives the expressions of gaps, where $n=8$, $k=4$, $i=1$:
\begin{equation*}
\begin{cases}
\nu_{\omega,1}=\varepsilon,&|\nu_{\omega,1}|=0,\\
\nu_{\omega,2}^{-1}=C_i(F_{k-2})=C_1(F_2)=baa,&|\nu_{\omega,2}|=-f_2=-3<0.
\end{cases}
\end{equation*}

\vspace{0.2cm}

\noindent\textbf{Theorem 2.4(4)}

\vspace{-0.6cm}
\begin{equation*}
\omega\in\textrm{T2.1}\Rightarrow\begin{cases}
\nu_{\omega,1}=s_{k+1}[n-f_k-i+1,f_{k+1}-i],&|\nu_{\omega,1}|=f_{k+2}-n>0,\\
\nu_{\omega,2}=s_{k-1}[n-f_k-i+1,f_{k-1}-i],&|\nu_{\omega,2}|=f_{k+1}-n>0.
\end{cases}
\end{equation*}

\begin{proof} Since $\omega$ is in T2.4, by Theorem 2.3,
$$\omega=s_{k-1}[f_{k-1}-i+1,f_{k-1}]\underline{s_k}s_{k-1}[1,n-f_k-i],$$
where $0\leq i\leq n-f_k$,
$\mu_1=s_{k-1}[f_{k-1}-i+1,f_{k-1}]$ and $\mu_2=s_{k-1}[1,n-f_k-i]$.

(1) Since $\nu_{s_k,1}=s_{k+1}$:
\begin{equation*}
\begin{split}
&\nu_{\omega,1}=\mu_2^{-1}\nu_{s_k,1}\mu_1^{-1}=s^{-1}_{k-1}[1,n-f_k-i]s_{k+1}s^{-1}_{k-1}[f_{k-1}-i+1,f_{k-1}]\\
=&s^{-1}_{k-1}[1,n-f_k-i]s_{k-1}s_{k-2}s_{k-1}s^{-1}_{k-1}[f_{k-1}-i+1,f_{k-1}]\\
=&s_{k-1}[n-f_k-i+1,f_{k-1}]s_{k-2}s_{k-1}[1,f_{k-1}-i]=s_{k+1}[n-f_k-i+1,f_{k+1}-i],
\end{split}
\end{equation*}
which yields $\nu_{\omega,1}=s_{k+1}[n-f_k-i+1,f_{k+1}-i]$, $|\nu_{\omega,1}|=f_{k+2}-n>0$.

(2) Since $\nu_{s_k,2}=s_{k-1}$:
\begin{equation*}
\begin{split}
&\nu_{\omega,2}=\mu_2^{-1}\nu_{s_k,2}\mu_1^{-1}=s^{-1}_{k-1}[1,n-f_k-i]s_{k-1}s^{-1}_{k-1}[f_{k-1}-i+1,f_{k-1}]\\
=&s_{k-1}[n-f_k-i+1,f_{k-1}-i],
\end{split}
\end{equation*}
i.e., $\nu_{\omega,2}=s_{k-1}[n-f_k-i+1,f_{k-1}-i]$, $|\nu_{\omega,2}|=f_{k+1}-n>0$.
\end{proof}

\noindent\textbf{Example.} Taking $\omega=ababaabab\in$T2.1, it appears in $F_\infty$ as:

$F_\infty=abaababaaba(ababaabab)aabaababaaba(ababaabab)aaba(ababaabab)aabaababaaba(\cdots$

Theorem 2.4 gives the expressions of gaps, where $n=9$, $k=4$, $i=1$:
\begin{equation*}
\begin{cases}
\nu_{\omega,1}=s_{k+1}[n-f_k-i+1,f_{k+1}-i]=s_5[1,12]=aabaababaaba,\\
|\nu_{\omega,1}|=f_{k+2}-n=f_6-9=12>0,\\
\nu_{\omega,2}=s_{k-1}[n-f_k-i+1,f_{k-1}-i]=s_3[1,4]=aaba,\\
|\nu_{\omega,2}|=f_{k+1}-n=f_5-9=4>0.
\end{cases}
\end{equation*}

\vspace{0.2cm}

\noindent\textbf{Theorem 2.4(5)}

\vspace{-0.6cm}
\begin{equation*}
\omega\in\textrm{T2.2}\Rightarrow\begin{cases}
\nu_{\omega,1}=s_k[n-f_{k-1}-i+1,f_k-i],&|\nu_{\omega,1}|=f_{k+1}-n>0,\\
\nu_{\omega,2}^{-1}=F_{k+1}[f_k-i,n-i-1],&|\nu_{\omega,2}|=f_k-n<0.
\end{cases}
\end{equation*}

\begin{proof} Since $\omega$ is in T2.2, by Theorem 2.3,
$$\omega=s_k[f_k-i+1,f_k]\underline{s_{k-1}}s_k[1,n-f_{k-1}-i],$$
where $0\leq i\leq n-f_{k-1}$,
$\mu_1=s_k[f_k-i+1,f_k]$ and $\mu_2=s_k[1,n-f_{k-1}-i]$.

(1) Since $\nu_{s_{k-1},1}=s_k$:
\begin{equation*}
\begin{split}
&\nu_{\omega,1}=\mu_2^{-1}\nu_{s_{k-1},1}\mu_1^{-1}=s^{-1}_k[1,n-f_{k-1}-i]s_ks^{-1}_k[f_k-i+1,f_k]\\
=&s_k[n-f_{k-1}-i+1,f_k-i],
\end{split}
\end{equation*}
i.e., $\nu_{\omega,1}=s_k[n-f_{k-1}-i+1,f_k-i]$, $|\nu_{\omega,1}|=f_{k+1}-n>0$.

(2) Since $\nu_{s_{k-1},2}=s_{k-2}$:
\begin{equation*}
\begin{split}
&\nu^{-1}_{\omega,2}=\mu_1\nu^{-1}_{s_{k-1},2}\mu_2=s_k[f_k-i+1,f_k]s^{-1}_{k-2}s_k[1,n-f_{k-1}-i].
\end{split}
\end{equation*}

We are going to determine the expression of $\nu^{-1}_{\omega,2}$, which we divide into three distinct cases.

(2.1) If $s_{k-2}$ is the suffix of $s_k[f_k-i+1,f_k]$, then
$$|s_k[f_k-i+1,f_k]|=i\geq f_{k-2}.$$
\begin{equation*}
\begin{split}
&\nu^{-1}_{\omega,2}=(s_{k-2}s_{k-3})[f_k-i+1,f_{k-1}]s_k[1,n-f_{k-1}-i]\\
=&(\alpha F_{k-2}F_{k-3}\beta^{-1})[f_k-i+1,f_{k-1}](\beta F_k\alpha^{-1})[1,n-f_{k-1}-i]\\
=&F_{k-1}[f_k-i,f_{k-1}]F_k[1,n-f_{k-1}-i-1]=(F_{k-1}F_k)[f_k-i,n-i-1]\\
=&(F_{k+1}\beta^{-1}\alpha^{-1}\beta\alpha)[f_k-i,n-i-1]=F_{k+1}[f_k-i,n-i-1].
\end{split}
\end{equation*}

(2.2) If $s_{k-2}$ is the prefix of $s_k[1,n-f_{k-1}-i]$, then
$$|s_k[1,n-f_{k-1}-i]|=n-f_{k-1}-i\geq f_{k-2},$$
i.e., $i\leq n-f_k$.
\begin{equation*}
\begin{split}
&\nu^{-1}_{\omega,2}=s_k[f_k-i+1,f_k](s_{k-3}s_{k-2})[1,n-f_k-i]\\
=&(\beta F_k\alpha^{-1})[f_k-i+1,f_k](\alpha F_{k-1}\beta^{-1}\alpha^{-1}\beta)[1,n-f_k-i]\\
=&F_k[f_k-i,f_k]F_{k-1}[1,n-f_k-i-1]=F_{k+1}[f_k-i,n-i-1].
\end{split}
\end{equation*}

(2.3) If $s_{k-2}$ is neither the suffix of $s_k[f_k-i+1,f_k]$, nor the prefix of $s_k[1,n-f_{k-1}-i]$,
then $n-f_k<i<f_{k-2}$, i.e., $n-f_k<f_{k-2}$, and $|s_k[f_k-i+1,f_k]|<f_{k-2}$, $|s_k[1,n-f_{k-1}-i]|<f_{k-2}$.
\begin{equation*}
\begin{split}
&\nu^{-1}_{\omega,2}=s_{k-2}[f_{k-2}-i+1,f_{k-2}]s^{-1}_{k-2}s_{k-2}[1,n-f_{k-1}-i]\\
=&s_{k-2}[f_{k-2}-i+1,f_{k-2}]s^{-1}_{k-2}[n-f_{k-1}-i+1,f_{k-2}]=s_{k-2}[f_{k-2}-i+1,n-f_{k-1}-i]\\
=&(s_{k-2}s_{k-3}s_{k-2})[f_{k-2}+f_{k-1}-i+1,n-f_{k-1}+f_{k-1}-i]\\
=&s_{k-1}[f_k-i+1,n-i]=F_{k+1}[f_k-i,n-i-1].
\end{split}
\end{equation*}

According to (2.1)-(2.3), we know $\nu_{\omega,2}^{-1}=F_{k+1}[f_k-i,n-i-1]$, $|\nu_{\omega,2}|=f_k-n<0$.
\end{proof}

\noindent\textbf{Example.} Taking $\omega=baabaabab\in$T2.2, it appears in $F_\infty$ as:

$F_\infty=abaaba(baabaabab)aaba(baabaaba(b)aabaabab)aaba(baabaabab)aaba(baabaaba(b)a\cdots$\\
where "$(b)$" in "$baabaaba(b)aabaabab$" shows a overlap.

Theorem 2.4 gives the expressions of gaps, where $n=9$, $k=4$, $i=1$:
\begin{equation*}
\begin{cases}
\nu_{\omega,1}=s_k[n-f_{k-1}-i+1,f_k-i]=s_4[4,7]=aaba,\\
|\nu_{\omega,1}|=f_{k+1}-n=13-9=4>0,\\
\nu_{\omega,2}^{-1}=F_{k+1}[f_k-i,n-i-1]=F_5[7,7]=b,\\
|\nu_{\omega,2}|=f_k-n=f_4-9=-1<0.
\end{cases}
\end{equation*}

\vspace{0.2cm}

\noindent\textbf{Theorem 2.4(6)}

\vspace{-0.6cm}
\begin{equation*}
\omega\in\textrm{T2.3}\Rightarrow\begin{cases}
\nu_{\omega,1}^{-1}=s_{k-1}[f_{k+1}-n-i,f_{k-1}-i-1],&|\nu_{\omega,1}|=f_k-n<0,\\
\nu_{\omega,2}^{-1}=F_k[f_{k+1}-n-i-1,f_k-i-2],&|\nu_{\omega,2}|=f_{k-1}-n<0.
\end{cases}
\end{equation*}

\begin{proof} Since $\omega$ is in T2.3, by Theorem 2.3,
$$\omega=s_{k-1}[f_{k+1}-n-i,f_{k-1}]\underline{s_{k-2}}s_{k-1}[1,f_{k-1}-i-1],$$
where $0\leq i\leq f_{k+1}-n-2$,
$\mu_1=s_{k-1}[f_{k+1}-n-i,f_{k-1}]$ and $\mu_2=s_{k-1}[1,f_{k-1}-i-1]$.

(1) Since $\nu_{s_{k-2},1}=s_{k-1}$:
\begin{equation*}
\begin{split}
&\nu^{-1}_{\omega,1}=\mu_1\nu^{-1}_{s_{k-2},1}\mu_2=s_{k-1}[f_{k+1}-n-i,f_{k-1}]s^{-1}_{k-1}s_{k-1}[1,f_{k-1}-i-1]\\
=&s_{k-1}[f_{k+1}-n-i,f_{k-1}]s^{-1}_{k-1}[f_{k-1}-i,f_{k-1}]=s_{k-1}[f_{k+1}-n-i,f_{k-1}-i-1],
\end{split}
\end{equation*}
i.e., $\nu_{\omega,1}^{-1}=s_{k-1}[f_{k+1}-n-i,f_{k-1}-i-1]$, $|\nu_{\omega,1}|=f_k-n<0$.

(2) Since $\nu_{s_{k-2},2}=s_{k-3}$:
\begin{equation*}
\begin{split}
&\nu^{-1}_{\omega,2}=\mu_1\nu^{-1}_{s_{k-2},2}\mu_2=s_{k-1}[f_{k+1}-n-i,f_{k-1}]s^{-1}_{k-3}s_{k-1}[1,f_{k-1}-i-1].
\end{split}
\end{equation*}

We are going to determine the expression of $\nu^{-1}_{\omega,2}$, which we divide into three distinct cases.

(2.1) If $s_{k-3}$ is the suffix of $s_{k-1}[f_{k+1}-n-i,f_{k-1}]$, then
$$|s_{k-1}[f_{k+1}-n-i,f_{k-1}]|=f_{k-1}-f_{k+1}+n+i+1\geq f_{k-3},$$
i.e., $i\geq f_k+f_{k-3}-n-1$.
\begin{equation*}
\begin{split}
&\nu^{-1}_{\omega,2}=(s_{k-3}s_{k-4})[f_{k+1}-n-i,f_{k-2}]s_{k-1}[1,f_{k-1}-i-1]\\
=&(\beta F_{k-3}F_{k-4}\alpha^{-1})[f_{k+1}-n-i,f_{k-2}](\alpha F_{k-1}\beta^{-1})[1,f_{k-1}-i-1]\\
=&F_{k-2}[f_{k+1}-n-i-1,f_{k-2}]F_{k-1}[1,f_{k-1}-i-1]=F_k[f_{k+1}-n-i-1,f_k-i-2].
\end{split}
\end{equation*}

(2.2) If $s_{k-3}$ is the prefix of $s_{k-1}[1,f_{k-1}-i-1]$, then
$$|s_{k-1}[1,f_{k-1}-i-1]|=f_{k-1}-i-1\geq f_{k-3},$$
i.e., $i\leq f_{k-2}-1$.
\begin{equation*}
\begin{split}
&\nu^{-1}_{\omega,2}=s_{k-1}[f_{k+1}-n-i,f_{k-1}](s_{k-4}s_{k-3})[1,f_{k-2}-i-1]\\
=&(\alpha F_{k-1}\beta^{-1})[f_{k+1}-n-i,f_{k-1}](\beta F_{k-4}F_{k-3}\beta^{-1})[1,f_{k-2}-i-1]\\
=&F_{k-1}[f_{k+1}-n-i,f_{k-1}]F_{k-2}[1,f_{k-2}-i-2]=F_k[f_{k+1}-n-i-1,f_k-i-2],
\end{split}
\end{equation*}
the 3-rd equality holds because $F_{k-4}F_{k-3}=F_{k-2}\alpha^{-1}\beta^{-1}\alpha\beta$.

(2.3) If $s_{k-3}$ is neither the suffix of $s_{k-1}[f_{k+1}-n-i,f_{k-1}]$, nor the prefix of $s_{k-1}[1,f_{k-1}-i-1]$.
Then $f_{k-2}-1<i<f_k+f_{k-3}-n-1$, i.e., $f_{k-2}-1<f_k+f_{k-3}-n-1$, so $n<f_k+f_{k-3}-f_{k-2}<f_k$, which is obviously not true. So this case does not exist.

According to (2.1)-(2.3), we have $\nu_{\omega,2}^{-1}=F_k[f_{k+1}-n-i-1,f_k-i-2]$, $|\nu_{\omega,2}|=f_{k-1}-n<0$.
\end{proof}

\noindent\textbf{Example.} Taking $\omega=baababaab\in$T2.3, it appears in $F_\infty$ as:

$F_\infty=a(baababaa(b)aaba(baab)abaa(b)aababaa(b)aaba(baab)abaa(b)aaba(baab)abaa(b)a\cdots$\\
where "$(b)$" in "$baababaa(b)aaba$" and "$(baab)$" in "$aaba(baab)abaa$" show two different overlaps.

Theorem 2.4 gives the expressions of gaps, where $n=9$, $k=4$, $i=1$:
\begin{equation*}
\begin{cases}
\nu_{\omega,1}^{-1}=s_{k-1}[f_{k+1}-n-i,f_{k-1}-i-1]=s_3[3,3]=b,\\
|\nu_{\omega,1}|=f_k-n=f_4-9=-1<0,\\
\nu_{\omega,2}^{-1}=F_k[f_{k+1}-n-i-1,f_k-i-2]=F_4[2,5]=baab,\\
|\nu_{\omega,2}|=f_{k-1}-n=f_3-9=-4<0.
\end{cases}
\end{equation*}


\vspace{0.5cm}

\stepcounter{section}

\noindent\textbf{\large{6.~Combinatorial Properties Of Factors}}

\vspace{0.4cm}

In this section, we will discuss some combinatorial properties of the factors of the Fibonacci sequence.
Let $F_\infty=x_1x_2\cdots x_n\cdots$
be the Fibonacci sequence and $\omega\prec F_\infty$ be a factor of $F_\infty$, as usual let $\{\omega_p\}_{p\ge 1}$ be the factor sequence where $\omega_p$ is $p$-th appearance of $\omega$. Notice that if we consider the location of $\omega_p\in F_\infty$, then the factors $\omega_p$ and $\omega_q(p\neq q)$ are distinct.
In fact, $\omega_p$ should be regarded as two variables $\omega$ and $p$, $\omega$ is the factor and $p$ indicates the location of $\omega$.

Let $\cal P$ be a property, we say $\omega\in{\cal P}$ if there exists an index $p\in\mathbb N$ such that $\omega_p\in{\cal P}$. But in this case, we do not know where is the location of $\omega_p\in F_\infty$, so we wish to determine the set $\{(\omega, p),\omega\prec F_\infty, p\in\mathbb{N}|~\omega_p\in{\cal P}\}.$

\vspace{0.2cm}

\noindent\textbf{Example.} we consider ${\cal P}$ is "property of square factor", that is, if $\omega\in{\cal P}$, then there exists $p$ such that $\omega_p\omega_{p+1}\in F_\infty$.
 Let $\omega=ab\in F_\infty$, then $\omega_1=F_\infty[1,2]$, $\omega_2=F_\infty[4,5]$, and we know that $\omega_1\not\in {\cal P}$, $\omega_2\in {\cal P}$, $\omega\in {\cal P}$.

\vspace{0.2cm}

By the example, we are led naturally to study combinatorial properties of factor $\omega$ of the following two types:

\vspace{0.2cm}

\noindent\textbf{Local Question: Determine all factors $\omega\in F_\infty$ such that $\omega\in {\cal P}$,
i.e., there exists $p\in\mathbb N$ such that $\omega_p\in{\cal P}.$}

\vspace{0.1cm}

\noindent\textbf{Global Question: Determine all factors $\omega\in F_\infty$ and all indices $p$ such that $\omega_p\in {\cal P}$.}

\vspace{0.2cm}

More precisely, define the spectrum of $\cal P$ by
$$\Lambda({\cal P}):=\Lambda({\cal P})(\omega,p):=\{(\omega, p),\omega\prec F_\infty, p\in\mathbb{N}|~\omega_p\in{\cal P}\}.$$
By the definition above, the Global question is equivalent to determine the spectrum of the property $\cal P$.

\vspace{0.2cm}

\noindent\textbf{Remark.} By the definition above, the spectrum of the property $\cal P$
depends two independent variables $\omega$ and $p$.
For a given factor $\omega$, the spectrum $\Lambda({\cal P})$ will give all indices $p$ such that $\omega_p\in {\cal P}$; and for a given index
$p$, the spectrum $\Lambda({\cal P})$ will give all factors $\omega$ such that $\omega_p\in {\cal P}$. The Local question is equivalent to determine the projection
of the spectrum $\Lambda({\cal P})$ on factor space.

\vspace{0.2cm}

We will study mainly some combinatorial properties such as "adjacent property, separated property and overlapped property of factors" for both questions.
From our knowledge, all previous studies on combinatorial over words concern with only local question,
in fact, "Global Question" is much more  difficult than "Local Question".

\vspace{0.2cm}

\noindent\textbf{Notation.} For the convenience for the discussions below, we give some notations.

(1)$\Gamma_a=\{p\prec\mathbb{N}|~F_\infty[p]=a\}$; (2)$\Gamma_b=\{p\prec\mathbb{N}|~F_\infty[p]=b\}$;

(3)$\Gamma_{aa}=\{p\prec\mathbb{N}|~F_\infty[p]=a,F_\infty[p+1]=a\}$; (4)$\Gamma_{ab}=\{p\prec\mathbb{N}|~F_\infty[p]=a,F_\infty[p+1]=b\}$.

It is easy to see that $\Gamma_a\sqcup\Gamma_b=\mathbb{N}$ and $\Gamma_{aa}\sqcup\Gamma_{aa}=\Gamma_a$.

\vspace{0.2cm}

\noindent\textbf{Remark.} It is known that $F_\infty[p]$ can be expressed explicitly by the following formula:
$F_\infty[p]=a$ if $[(p+1)\xi]-[p\xi]=0$, $F_\infty[p]=b$ if $[(p+1)\xi]-[p\xi]=1$, where $\xi=\frac{3-\sqrt{5}}{2}$.
So the sets $\Gamma_a$ and $\Gamma_b$ can be define easily.

\begin{lemma} Let $\omega\in F_\infty$, then
$\nu_{\omega,p}=\nu_{\omega,1}\Leftrightarrow p\in\Gamma_a,~ \nu_{\omega,p}=\nu_{\omega,2}\Leftrightarrow p\in\Gamma_b.$
\end{lemma}

\begin{proof}
By Theorem 2.2, the gap sequence $\{\nu_{\omega,p}\}_{p\geq1}$ is Fibonacci sequence, in which $\nu_{\omega,1}$ and $\nu_{\omega,2}$ correspond with letter
$a$ and $b$ respectively. Thus $\nu_{\omega,p}=\nu_{\omega,1}\Leftrightarrow F_\infty[p]=a$ and $\nu_{\omega,p}=\nu_{\omega,2}\Leftrightarrow F_\infty[p]=b$.
\end{proof}

\begin{definition}[Power Property]
Let $\omega\in F_\infty$. We say that $\omega\in{\cal P}_i(i\ge 1)$ if there exists $p$ such that
$\omega_p\cdots\omega_{p+i}\prec F_\infty$.
\end{definition}

\begin{proposition} [Global property for Power Property]\

(1) $\Lambda({\cal P}_1)=(\emph{T}1.2,\Gamma_b)\sqcup(\emph{T}1.3,\Gamma_a)$;

(2) $\Lambda({\cal P}_2)=(\emph{T}1.3,\Gamma_{aa})$;

(3) $\Lambda({\cal P}_2)\setminus\Lambda({\cal P}_1)=(\emph{T}1.2,\Gamma_b)\sqcup(\emph{T}1.3,\Gamma_{ab})$;

(4) $\Lambda({\cal P}_i)=\emptyset$, $i\geq3$.
\end{proposition}

\begin{proof}

(1) The spectrum of ${\cal P}_1$ contains all $\omega$ and $p$ such that $\omega_p\in\mathcal{P}_1$, i.e., $\omega_p\omega_{p+1}\prec F_\infty$, which is equal to $\nu_{\omega,p}=\varepsilon$.
By Theorem 2.2, there are two cases:

Case 1: $\nu_{\omega,1}=\varepsilon$ and $\nu_{\omega,p}=\nu_{\omega,1}$.
By Theorem 2.4, $\nu_{\omega,1}=\varepsilon\Leftrightarrow\omega\in$T1.3.
By Lemma 6.1, $\nu_{\omega,p}=\nu_{\omega,1}\Leftrightarrow p\in\Gamma_a$.

Case 2: $\nu_{\omega,2}=\varepsilon$ and $\nu_{\omega,p}=\nu_{\omega,1}$.
By Theorem 2.4, $\nu_{\omega,2}=\varepsilon\Leftrightarrow\omega\in$T1.2.
By Lemma 6.1, $\nu_{\omega,p}=\nu_{\omega,2}\Leftrightarrow p\in\Gamma_b$.

(2) The spectrum of ${\cal P}_2$ contains all $\omega$ and $p$ such that $\omega_p\omega_{p+1}\omega_{p+2}\prec F_\infty$. By Theorem 2.2 and Theorem 2.4, it is equivalent to $\nu_{\omega,p}=\nu_{\omega,p+1}=\varepsilon$, i.e., $\nu_{\omega,1}=\varepsilon$ and $aa\prec F_\infty$ or $\nu_{\omega,2}=\varepsilon$ and $bb\prec F_\infty$.
Since $aa\prec F_\infty$ and $bb\not\prec F_\infty$, the spectrum of ${\cal P}_2$ contains $\omega_p$ with $\omega\in$T1.3, $F_\infty[p]=a$ and $F_\infty[p+1]=a$, i.e., $(\textrm{T}1.3,\Gamma_{aa})$.

(3) Since (1) and (2), by the minus of sets, $\Lambda({\cal P}_2)\setminus\Lambda({\cal P}_1)=(\textrm{T}1.2,\Gamma_b)\sqcup(\textrm{T}1.3,\Gamma_{ab})$.

(4) The spectrum of ${\cal P}_3$ contains all $\omega$ and $p$ such that
$\omega_p\omega_{p+1}\omega_{p+2}\omega_{p+3}\prec F_\infty$. By Theorem 2.2 and Theorem 2.4, it is equivalent to $\nu_{\omega,p}=\nu_{\omega,p+1}=\nu_{\omega,p+2}=\varepsilon$, i.e., $\nu_{\omega,1}=\varepsilon$ and $aaa\prec F_\infty$ or $\nu_{\omega,2}=\varepsilon$ and $bbb\prec F_\infty$. Since $aaa\not\prec F_\infty$ and $bbb\not\prec F_\infty$, both of them are obviously not true. So there is no 4-square word in $F_\infty$, i.e.$\Lambda({\cal P}_3)=\emptyset$. Similarly, $\Lambda({\cal P}_i)=\emptyset$, $i\geq3$.
\end{proof}

We have shown that the Local question is equivalent to determine the projection
of the spectrum $\Lambda({\cal P})$ on factor space. Thus we get the Local property for power property immediately from Proposition 6.3, where only Corollary 6.3(3) is proved by the minus of sets.

\begin{corollary} [Local property for Power Property]\

(1) $\Lambda({\cal P}_1)=\emph{T}1.2\sqcup\emph{T}1.3$;

(2) $\Lambda({\cal P}_2)=\emph{T}1.3$;

(3) $\Lambda({\cal P}_2)\setminus\Lambda({\cal P}_1)=\emph{T}1.2$;

(4) $\Lambda({\cal P}_i)=\emptyset$, $i\geq3$.
\end{corollary}

\vspace{0.2cm}

\noindent\textbf{Remark.} Corollary 6.4(1) means "$\omega^2\prec F_\infty\Leftrightarrow\omega\in$T1.2$\sqcup$T1.3", i.e., "$\omega^2\prec F_\infty\Leftrightarrow\omega$ is a conjugation of $F_k$", which is equivalent to Theorem 3(1) to 3(3) in Wen and Wen\cite{WW1994}. Similarly, Corollary 6.4(2) is equivalent to Theorem 3(4), Corollary 6.4(3) is equivalent to Theorem 3(5), Corollary 6.4(4) is equivalent to Theorem 3(6) in Wen and Wen\cite{WW1994}.

\begin{definition}[Separated Properties]
Let $\omega\in F_\infty$. We say that $\omega\in{\cal S}_i~(i=1,2,\cdots)$ if there exists $p$ and nonempty factors $u_1,\cdots, u_{i-1}$ such that
$$\omega_p u_1\omega_{p+1}u_2\cdots u_{i-1}\omega_{p+i}\in F_{\infty}.$$
If all $i\in\mathbb N$, $\omega\in {\cal S}_i$, we say that $\omega\in{\cal S}_\infty.$

\end{definition}

\noindent\textbf{Remark.} By definition of $\nu_{\omega,p}$, $\omega\in{\cal S}_i$ is equivalent to there exists $p$ such that $|\nu_{\omega,p}|$, $|\nu_{\omega,p+1}|$, $\cdots,$ $|\nu_{\omega,p+i-1}|$ are strictly positive.

\begin{proposition}[Global property for separated Property]\

(1) $\Lambda(\mathcal{S}_1)=(\emph{T}1.1\sqcup\emph{T}2.1,\mathbb{N})\sqcup
(\emph{T}1.2\sqcup\emph{T}2.2,\Gamma_a)$;

(2) $\Lambda(\mathcal{S}_2)=(\emph{T}1.1\sqcup\emph{T}2.1,\mathbb{N})\sqcup(\emph{T}1.2\sqcup
\emph{T}2.2,\Gamma_{aa})$;

(3) $\Lambda(\mathcal{S}_3)=(\emph{T}1.1\sqcup\emph{T}2.1,\mathbb{N})$;

(4) $\Lambda(\mathcal{S}_\infty)=\emph{T}1.1\sqcup\emph{T}2.1$.

\end{proposition}

\begin{proof} (1) $\omega_p\in\mathcal{S}_1$ means $|\nu_{\omega,p}|>0$. By Theorem 2.2 and Theorem 2.4, there are two cases:

Case 1: Both $|\nu_{\omega,1}|$ and $|\nu_{\omega,2}|$ are strictly positive. Then $\omega\in$T1.1$\sqcup$T2.1.

Case 2: $|\nu_{\omega,1}|>0$, $|\nu_{\omega,2}|\leq0$ and $\nu_{\omega,p}=\nu_{\omega,1}$. Then $\omega\in$T1.2$\sqcup$T2.2. By Lemma 6.1, $\nu_{\omega,p}=\nu_{\omega,1}\Leftrightarrow F_\infty[p]=a$, i.e., $p\in\Gamma_a$.

(2) $\omega_p\in\mathcal{S}_2$ means both $|\nu_{\omega,p}|$ and $|\nu_{\omega,p+1}|$ are strictly positive. By Theorem 2.2 and Theorem 2.4, there are two cases:

Case 1: Both $|\nu_{\omega,1}|$ and $|\nu_{\omega,2}|$ are strictly positive. Then $\omega\in$T1.1$\sqcup$T2.1.

Case 2: $|\nu_{\omega,1}|>0$, $|\nu_{\omega,2}|\leq0$ and $\nu_{\omega,p}=\nu_{\omega,p+1}=\nu_{\omega,1}$. Then $\omega\in$T1.2$\sqcup$T2.2. By Lemma 6.1, $\nu_{\omega,p}=\nu_{\omega,p+1}=\nu_{\omega,1}\Leftrightarrow F_\infty[p]=F_\infty[p+1]=a$, i.e., $p\in\Gamma_{aa}$.

(3) $\omega_p\in\mathcal{S}_3$ means $|\nu_{\omega,p}|$, $|\nu_{\omega,p+1}|$ and $|\nu_{\omega,p+2}|$ are all strictly positive. Since $aaa,~bbb\not\prec F_\infty$, both $\nu_{\omega,1}$ and $\nu_{\omega,2}$ can not appear three times continually.
So $\omega_p\in\mathcal{S}_3$ contains both $|\nu_{\omega,1}|$ and $|\nu_{\omega,2}|$ are strictly positive. By Theorem 2.4, $\omega\in$T1.1$\sqcup$T2.1.

(4) When $\omega\in$T1.1$\sqcup$T2.1., both $|\nu_{\omega,1}|$ and $|\nu_{\omega,2}|$ are strictly positive, thus $\forall p\in\mathbb{N}$, $|\nu_{\omega,p}|>0$, i.e., $\omega\in\mathcal{S}_\infty$.
\end{proof}

\begin{corollary} [Local property for separated properties]\

(1) $\Lambda(\mathcal{S}_1)=\emph{T}1.1\sqcup\emph{T}1.2\sqcup\emph{T}2.1\sqcup\emph{T}2.2$;

(2) $\Lambda(\mathcal{S}_2)=\Lambda(\mathcal{S}_1)$;

(3) $\Lambda(\mathcal{S}_3)=\emph{T}1.1\sqcup\emph{T}2.1$;

(4) $\Lambda(\mathcal{S}_\infty)=\Lambda(\mathcal{S}_3)$.
\end{corollary}

\noindent\textbf{Remark.} Corollary 6.7(4) determine the factor with separated property completely, which we only know T1.1 (singular word) before, see\cite{WW1994}.

\begin{definition}[Overlapped Property]
Let $\omega\in F_\infty$. We say that $\omega\in{\cal O}_i~(i=1,2,\cdots)$ if there exists $p$ and nonempty factors $u_1,\cdots, u_{i-1}$ such that:
$$\omega_p u_1^{-1}\omega_{p+1}u_2^{-1}\cdots u_{i-1}^{-1}\omega_{p+i}\in F_{\infty}.$$
If all $i\in\mathbb N$, $\omega\in {\cal O}_i$, we say that $\omega\in{\cal O}_\infty.$

\end{definition}

\noindent\textbf{Remark.} By definition of $\nu_{\omega,p}$, $\omega\in{\cal O}_i$ is equivalent to there exists $p$ such that $|\nu_{\omega,p}|$, $|\nu_{\omega,p+1}|$, $\cdots,$ $|\nu_{\omega,p+i-1}|$ are strictly negative.

\begin{proposition}[Global property for overlapped Property]\

(1) $\Lambda(\mathcal{O}_1)=(\emph{T}1.3\sqcup\emph{T}2.2,\Gamma_b)\sqcup(\emph{T}2.3,\mathbb{N})$;

(2) $\Lambda(\mathcal{O}_2)=(\emph{T}2.3,\mathbb{N})$;

(3) $\Lambda(\mathcal{O}_\infty)=\emph{T}2.3$.

\end{proposition}

\begin{proof} (1) $\omega_p\in\mathcal{O}_1$ means $|\nu_{\omega,p}|<0$.
By Theorem 2.2 and Theorem 2.4, there are two cases:

Case 1: Both $|\nu_{\omega,1}|$ and $|\nu_{\omega,2}|$ are strictly negative. Then $\omega\in$T2.3.

Case 2: $|\nu_{\omega,1}|\geq0$, $|\nu_{\omega,2}|<0$ and $\nu_{\omega,p}=\nu_{\omega,2}$.
Then $\omega\in$T1.3$\sqcup$T2.2. By Lemma 6.1, $\nu_{\omega,p}=\nu_{\omega,2}\Leftrightarrow F_\infty[p]=b$, i.e., $p\in\Gamma_b$.

(2) $\omega_p\in\mathcal{O}_2$ means both $|\nu_{\omega,p}|$ and $|\nu_{\omega,p+1}|$ are strictly negative. Since $bb\not\in F_\infty$, the word $\omega_p\in(\textrm{T}1.3\sqcup\textrm{T}2.2,\Gamma_b)$ doesn't possess $\mathcal{O}_2$.
When $\omega\in$T2.3, both $|\nu_{\omega,1}|$ and $|\nu_{\omega,2}|$ are strictly negative. So for $\forall p\in\mathbb{N}$, $\omega_p\in$T2.3 possesses $\mathcal{O}_2$.

(3) By Theorem 2.4, when $\omega\in$T2.3, both $|\nu_{\omega,1}|$ and $|\nu_{\omega,2}|$ are strictly negative, thus $\forall~p$, $|\nu_{\omega,p}|<0$, i.e., $\omega\in\mathcal{O}_\infty$.
\end{proof}

\begin{corollary}[Local property for overlapped property]\

(1) $\Lambda(\mathcal{O}_1)=\emph{T}1.3\sqcup\emph{T}2.2\sqcup\emph{T}2.3$;

(2) $\Lambda(\mathcal{O}_2)=\emph{T}2.3$;

(3) $\Lambda(\mathcal{O}_\infty)=\Lambda(\mathcal{O}_2)$.
\end{corollary}

\noindent\textbf{Remark.} Corollary 6.10 contains Theorem 6 in Wen and Wen\cite{WW1994}.
Moreover, we correct a small mistake (Lemma 7) there: If $\omega\in \mathcal{O}_1$, then the overlap of $\omega$ is unique.
In fact, when $\omega\in$T2.3, $\nu_{\omega,1}^{-1}=s_{k-1}[f_{k+1}-n-i,f_{k-1}-i-1]$, $|\nu_{\omega,1}|=f_k-n<0$ and $\nu_{\omega,2}^{-1}=F_k[f_{k+1}-n-i-1,f_k-i-2]$, $|\nu_{\omega,2}|=f_{k-1}-n<0$, which means the overlap of $\omega$ is not unique.
For instance, let $\omega=baababaab\in$T2.3, both $baababaa(b)aababaab$ and $baaba(baab)abaab$ are factors of $F_\infty$.


\vspace{0.5cm}

\noindent\textbf{\large{Acknowledgments}}

\vspace{0.4cm}

The research is supported by the Grant NSF No.61071066, No.11271223 and No.11371210.


\end{CJK*}
\end{document}